\documentclass{aic}
\usepackage{amsfonts,mathtools,enumerate,thmtools,thm-restate}
\hypersetup{%
	pdftitle={Extremal functions for sparse minors},
	pdfauthor={Hendrey--Norin--Wood}}
\usepackage[noabbrev,capitalise]{cleveref}
\usepackage[numbers,sort&compress]{natbib}
\makeatletter
\def\NAT@spacechar{~}
\makeatother
\allowdisplaybreaks
\usepackage[shortlabels]{enumitem}
\setlist[itemize]{topsep=0.2ex,itemsep=0ex,parsep=0.5ex}
\setlist[enumerate]{topsep=0.2ex,itemsep=0ex,parsep=0.5ex}
\theoremstyle{plain}
\newtheorem{thm}{Theorem}[section]
\newtheorem{lem}[thm]{Lemma}
\newtheorem{cor}[thm]{Corollary}

\theoremstyle{definition}
\newtheorem{example}{Example}
\crefname{lem}{Lemma}{Lemmas}
\crefname{thm}{Theorem}{Theorems}
\crefname{cor}{Corollary}{Corollaries}
\crefname{prop}{Proposition}{Propositions}
\crefname{conj}{Conjecture}{Conjectures}
\crefname{claim}{Claim}{Claims}
\crefname{example}{Example}{Examples}
\crefformat{equation}{(#2#1#3)}
\Crefformat{equation}{Equation #2(#1)#3}
\newcommand{\defn}[1]{\textcolor{purple}{\emph{#1}}}
\newcommand{\mathdefn}[1]{\textcolor{purple}{#1}}
\DeclarePairedDelimiter\ceil\lceil\rceil

\renewcommand{\geq}{\geqslant}
\renewcommand{\leq}{\leqslant}

\newcommand{\NN}{\mathbb{N}}

\newcommand{\RR}{\mathbb{R}}
\renewcommand{\v}{\textup{\textsf{v}}}
\newcommand{\e}{\textup{\textsf{e}}}
\renewcommand{\d}{\textup{\textsf{d}}}
\newcommand{\eps}{\varepsilon}
\newcommand{\cx}{c_{f}}
\newcommand{\mc}[1]{\mathcal{#1}}
\newcommand{\bb}[1]{\mathbb{#1}}
\newcommand{\bl}[1]{\boldsymbol{#1}}
\newcommand{\brm}[1]{\operatorname{#1}}

\newcommand{\fs}[2]{\left(\frac{#1}{#2}\right)}
\newcommand{\s}[1]{\left(#1\right)}
\newcommand{\ip}[2]{\left\langle #1, #2 \right\rangle}
\newcommand{\Vol}{\brm{Vol}}
\aicAUTHORdetails{%
	title = {Extremal functions for sparse minors},
	author = {Kevin Hendrey, Sergey Norin, and David R. Wood},
	plaintextauthor = {Kevin Hendrey, Sergey Norin, David R. Wood},
	keywords = {graph minor, extremal function},
}   %%% END \aicAUTHORdetails
\aicEDITORdetails{%
	year={2022},
	number={5},
	received={22 July 2021},   % received date: example: 7 January 2017
	published={25 July 2022},  % published date
	doi={10.19086/aic.2022.5},      % XXX = number of paper, e.g. aic006 for paper#6
}   %%% END \aicEDITORdetails
\begin{document}
\begin{frontmatter}[classification=text]
%%%%%%%%%%%%%%%%%%%%	
\author[kevin]{Kevin Hendrey\thanks{Supported by the Institute for Basic Science (IBS-R029-C1).}}
\author[sergey]{Sergey Norin\thanks{Supported by  NSERC under Discovery Grant No. $2017-05010$.}}
\author[david]{David R. Wood\thanks{Supported by the Australian Research Council.}}
%%%%%%%%%%%%%%%%%%%%	
\begin{abstract}
The \defn{extremal function} $c(H)$ of a graph $H$ is the supremum of densities of graphs not containing $H$ as a minor, where the \defn{density} of a graph $G$ is the ratio of the number of edges to the number of vertices.  Myers and Thomason~(2005), Norin, Reed, Thomason and Wood~(2020), and Thomason and Wales~(2019) determined the asymptotic behaviour of $c(H)$ for  all polynomially dense graphs $H$, as well as almost all graphs $H$ of constant density. 
			
We explore the asymptotic behavior of the extremal function in the regime not covered by the above results, where in addition to having constant density the graph $H$ is in a graph class admitting strongly sublinear separators. We  establish asymptotically tight bounds in many cases. For example, we prove that for every planar graph $H$, 
$$c(H) = (1+o(1))\cdot\max\left\{\frac{|V(H)|}{2},|V(H)| - \alpha (H)\right\},$$
extending recent results of Haslegrave, Kim and Liu~(2020). We also show that an asymptotically tight bound on the extremal function of graphs in minor-closed families proposed by Haslegrave, Kim and Liu~(2020) is equivalent to a well studied open weakening of Hadwiger's conjecture. 
\end{abstract}
\end{frontmatter}

%%%%%%%%%%%%%%%%%%%%%%
\section{Introduction}
\label{s:intro}

A graph $H$ is a \defn{minor} of a graph $G$, written $H\preceq G$, if a graph isomorphic to $H$ can be obtained from a subgraph of $G$ by contracting edges. The maximum density of graphs $G$ not containing a given graph $H$ as a minor has been actively investigated within extremal graph theory.

To formalize this notion, let $\v(G)$ and $\e(G)$ be the number of  vertices and edges of a graph $G$, respectively, and let $\d(G):=\frac{\e(G)}{\v(G)}$ be the \defn{density} of a non-null graph $G$. Following \citet{MyeTho05}, for a  graph $H$ with $\v(H) \geq 2$ we define the \defn{extremal function}  $c(H)$ of $H$ as the supremum of $\d(G)$ taken over all non-null graphs $G$ not containing $H$ as a minor.  

Mader~\cite{Mader67} proved that $c(H)$ is finite for every graph $H$. The exact value has been determined for various small graphs $H$. For example, if $K_t$ is the complete graph on $t\leq 9$ vertices, then $c(K_t)=t-2$ (see \citep{Dirac64,Mader68,Jorgensen94,ST06}); and if $P$ is the Petersen graph, then $c(P)=5$ (see \citep{HW18a}). Here we focus on asymptotic results for classes of graphs $H$.

The asymptotic behaviour of $c(K_t)$ was studied in~\cite{Kostochka82,Kostochka84,Thomason84}, and  was determined precisely by Thomason~\cite{Thomason01}, who showed that
\begin{equation}\label{e:Thomason}
c(K_t)=(\lambda+o(1))t\sqrt{\log{t}},
\end{equation}
where
$$\lambda = \max_{\alpha >0} \frac{1-e^{-\alpha}}{2\sqrt{\alpha}}=0.319\ldots.$$
Improving on results of~\cite{MyeTho05,ReeWoo15}, \citet{ThoWal19} recently  extended the upper bound from (\ref{e:Thomason}) to general graphs, by showing that for every graph $H$, 
\begin{equation}\label{e:Thomason2} 
c(H) \leq (\lambda + o_{\d(H)}(1))\,\v(H)\sqrt{\log \d(H)}.
\end{equation}

The inequality (\ref{e:Thomason2}) is tight in many regimes.  \citet{MyeTho05} showed that  it is tight (up to the choice of the error term) for almost all graphs with $n$ vertices and $n^{1 + \eps}$ edges for every fixed $\eps >0$, and for all regular graphs with these parameters. They also gave an explicit asymptotic formula for $c(H)$ for all such polynomially dense graphs.

Norin, Reed, Thomason and Wood~\cite{NRTW20} recently  showed that (\ref{e:Thomason2}) is also tight for almost all graphs of constant  density; that is, for almost all graphs $H$ with $\d(H)=d$, 
\begin{equation}
\label{e:Thomason3} 
c(H) \geq (\lambda - o_{d}(1))\,\v(H)\sqrt{\log d}.
\end{equation} 

\medskip
This paper investigates graph classes for which the extremal function behaves qualitatively differently, such as for the following concrete families:
\begin{itemize}
	\item Chudnovsky, Reed and Seymour~\cite{ChuReeSey11} proved that $c(K_{2,t})=\frac{t+1}{2}$ for all $t \geq 2$;
	\item Kostochka and Prince~\cite{KosPri10} proved that $c(K_{3,t})=\frac{t+3}{2}$ for all $t \geq 6300$;
	\item More generally, Myers~\cite{Myers03} considered the asymptotic behaviour of $c(K_{s,t})$ for fixed $s$ and $t$ and conjectured that $c(K_{s,t}) \leq c_s t$ for some constant independent on $t$. K\"uhn and Osthus~\cite{KuhOst05} and Kostochka and Prince~\cite{KosPri08}  independently proved this conjecture by showing that $c(K_{s,t}) =(\frac12+o_{s}(1))t$.
	\item Cs\'{o}ka et al.~\cite{CLNWY} proved that if $H$ is a disjoint union of cycles, then \begin{equation*}
	\label{e:cycledensity}
	c(H) \leq \frac{\v(H)+\brm{comp}(H)}{2}-1,
	\end{equation*}
	which is tight whenever every component of $H$ is an odd cycle.
\end{itemize}

All of the above families are structurally sparse and the extremal function is linear in the number of vertices. (In fact, $c(H)< (1+o(1))\,\v(H)$ for all these graphs.)

We show that this property generalizes to the large and well studied class of sparse graph families defined as follows. A graph family is \defn{monotone} if it is closed under taking subgraphs. A \defn{separation} of a graph $G$ is a pair $(A_1,A_2)$ of subsets of $V(G)$, such that $G=G[A_1]\cup G[A_2]$ and $A_1\setminus A_2\neq\emptyset$ and $A_2\setminus A_1\neq\emptyset$. A separation $(A_1,A_2)$ has \defn{order} $|A_1\cap A_2|$. A separation $(A_1,A_2)$ is \defn{balanced} if $|A_1|,|A_2|\geq \frac{\v(G)}{3}$. A graph family $\mc{F}$ admits \defn{strongly sublinear separators} (written $\mc{F}$ is \defn{s.s.s.}, for brevity) if $\mc{F}$ is monotone, and  there exists $\beta<1$ and $c>0$ such that every graph $G\in\mc{F}$ has a balanced separation of order at most $c\,\v(G)^\beta$. For example, every proper minor-closed family\footnote{A family of graphs is  \defn{proper minor-closed} if it is closed under isomorphisms and taking minors, and does not include all graphs.} is s.s.s., as proved by \citet{AST90} with $\beta=\frac12$. More generally, every family with polynomial expansion is s.s.s.\ \citep{DN16}.

For a graph $H$, let \defn{$\chi(H)$} be the chromatic number of $H$, and let \defn{$\alpha(H)$} be the independence number of $H$. A \defn{vertex cover} of $H$ is a set $S\subseteq V(H)$ such that $H-S$ has no edges. Let \defn{$\tau(H)$} be the minimum size of a vertex cover of $H$. Note that $\tau(H)=\v(H)-\alpha(H)$. For $k\in\NN$, let \defn{$\alpha_k(H)$} be the maximum $|X|$ such that $X \subseteq V(H)$ and $\chi(H[X]) \leq k$. Thus $\alpha_1(H) = \alpha(H)$. Let \defn{$k\,H$} denote the union of $k$ vertex-disjoint copies of $H$.
		
Before formally stating our first main result,  we describe two natural lower bounds on $c(H)$. First, since $H$ is not a minor of $K_{\v(H)-1}$,
\begin{equation}
\label{NaiveLowerBound}
c(H) \geq \d (K_{\v(H)-1}) = \frac{\v(H)}{2} -1.
\end{equation}
For the second bound, observe that $\tau(H) \leq \tau(G)$ whenever $H$ is a minor of $G$. It follows that $H$  is not a minor of the complete bipartite graph $K_{\tau(H)-1,n}$ for any $n$. It follows that 
\begin{equation}\label{e:naive2} 
c(H) \geq \lim_{n \to \infty} \d(K_{\tau(H)-1,n}) = \tau(H)-1.
\end{equation}
We show that these lower bounds are asymptotically tight for $4$-colourable graphs in s.s.s.\ families. 

\begin{thm}\label{t:main0}
For every s.s.s.\ family $\mathcal{F}$ and for every $H \in \mc{F}$ with $\chi(H) \leq 4$,
\begin{equation}\label{e:main0}
c(H) = (1+o_{\mc{F}}(1)) \cdot \max \s{\frac{\v(H)}{2},\tau(H)},
\end{equation}
where the error term $o_{\mc{F}}(1)$ depends on $\mc{F}$ and satisfies  $o_{\mc{F}}(1)\to 0$ as $\v(H)\to\infty$.
\end{thm}

Note that some assumption about $\chi(H)$ in \cref{t:main0} is necessary. In particular, the $t=10$ case in \cref{InterestingExample} (in \cref{s:fracintro}) shows that there exist s.s.s.\ families $\mc{F}$ containing arbitrarily large graphs $H$ with $\chi(H)=11$ that do not satisfy \eqref{e:main0}. 

\citet{HKL20} recently and independently proved\footnote{\citet{HKL20} state \cref{e:AlphaTwoUpperBound} for proper minor-closed families, but their proof holds for s.s.s.\ families.} that for every s.s.s.\ family $\mathcal{F}$ and every $H\in \mathcal{F}$,
\begin{equation}
\label{e:AlphaTwoUpperBound}
c(H) \leq  (1 + o_{\mathcal{F}}(1)) \left( \v(H) - \frac{\alpha_2(H)}{2} \right).
\end{equation}
Our methods developed for proving \cref{t:main0} readily lead to the following result that improves\footnote{\cref{t:AlphaThreeUpperBound} strengthens \cref{e:AlphaTwoUpperBound} since $\alpha_3(H)\geq\alpha_2(H)$ and $\tau(H)=\v(H)-\alpha(H)\leq \v(H)-\frac12\alpha_2(H)$.
For example, let $J$ be a complete 3-partite graph $K_{2b,b,b}$. Then $\v(J)=4b$, $\alpha_2(J) = 3b$, $\alpha_3(J)=4b$ and $\tau(H)=2b$. Let $\mc{F}$ consist of all subgraphs of unions of vertex-disjoint copies of $J$. Then $\mc{F}$ is s.s.s.\ and $k\,J\in \mc{F}$ for each $k\in \NN$. \cref{t:AlphaThreeUpperBound} implies $c(k\,J)  \leq (2+o(1))bk$, whereas
\cref{e:AlphaTwoUpperBound} only gives $c(k\,J)  \leq (\frac{5}{2}+o(1))bk$.} upon the upper bound in \cref{e:AlphaTwoUpperBound}.
\begin{restatable}{thm}{AlphaThreeUpperBound}\label{t:AlphaThreeUpperBound} 
	For every s.s.s.\ family $\mathcal{F}$ and every $H\in \mathcal{F}$,
	$$c(H)\leq (1+o_{\mathcal{F}}(1)) \max\s{ \v(H)-\frac{\alpha_3(H)}{2}, \tau(H) }.$$
\end{restatable}

The methods used in \cite{HKL20} are substantially different from ours and extend from minors to subdivisions, while ours do not. On the other hand, we are able to gain finer understanding of the extremal function for graphs in some s.s.s.\ families. 

To see this, consider a planar graph $H$. By the Four Color Theorem~\citep{RSST97}, $\tau(H)\leq\frac{3}{4}\v(H)$ and $\alpha_2(H)\geq\frac{1}{2}\v(H)$. Thus, 
\cref{t:main0} or \cref{e:AlphaTwoUpperBound} imply that $c(H) \leq (\frac{3}{4}+o(1))\,\v(H)$, where unions of disjoint copies of $K_4$ give tight examples by \eqref{e:naive2}. As noted by Haslegrave, Kim and Liu~\cite{HKL20}, this gives an asymptotic answer to the question of Reed and Wood~\cite{ReeWoo15} asking for the maximum of the ratio $\frac{c(H)}{\v(H)}$ for planar graphs. \cref{t:main0} gives more information by asymptotically determining the extremal function for all planar graphs.

Reed and Wood~\cite{ReeWoo15} also posed the problem of determining the maximum of the ratio $\frac{c(H)}{\v(H)}$ taken over  graphs $H$ with no $K_t$ minor.  Haslegrave, Kim and Liu~\cite[Problem F]{HKL20} more generally asked to determine the minimum $\rho_{\mc{F}}$ such that $c(H) \leq (\rho_{\mc{F}} + o(1))\,\v(H)$ for every minor closed graph family $\mc{F}$.  We give an explicit formula for $\rho_{\mc{F}}$ under a technical assumption, in the more general setting of s.s.s.\ families. 

\begin{thm}\label{t:upper0}
	Let $\mathcal{F}$ be a  s.s.s.\ graph family and let $$r_{\mathcal{F}} := \sup \left\{\frac{\tau(H)}{\v(H)} \: | \: k\,H \in \mc{F} \; \text{for every} \; k \in \NN  \right\} .$$ If $r_{\mc{F}} \geq \frac{2}{3}$ then for every graph  $H \in \mc{F}$, 
	$$c(H) \leq \s{r_{\mc{F}} + o(1)}\v(H),$$
	and there exists arbitrarily large graphs $H \in \mc{F}$ for which equality holds.
\end{thm}

Hadwiger's famous conjecture~\cite{Had43} states that $\chi(H) \leq t$ for every graph $H$ such that $K_{t+1} \not \preceq H$. As observed in~\cite{HKL20}, if Hadwiger's conjecture holds, then $r_{\mc{F}}=1-\frac{1}{t}$ for every minor-closed family $\mc{F}$ closed under disjoint union such that $K_{t} \in \mc{F}$ and $K_{t+1} \not\in \mc{F}$ and $t \geq 2$.  

Hadwiger's conjecture implies the following 
\begin{description} \item[$\star(t)$] 
 $\alpha(H) \geq \frac{\v(H)}{t}$ for every graph $H$ such that $K_{t+1} \not \prec H$.
\end{description}
This well studied weakening of Hadwiger's conjecture is still open, see  \cite[Section~4]{Sey16Survey}  for a detailed overview of known results.
Interestingly, \cref{t:upper0} implies  that for the family $\mc{F}$ of all $K_{t+1}$-minor free graphs, $r_{\mc{F}} \leq 1-\frac{1}{t}$ if and only if $\star(t)$ holds. (See \cref{thm:minorclosed} for details.)\ Thus an explicit answer to the above-mentioned questions from~\cite{ReeWoo15,HKL20}  would require one to first resolve the validity of this weakening of Hadwiger's conjecture. 

We prove Theorems~\ref{t:main0}--\ref{t:upper0} by showing that the extremal function in s.s.s.\ graph classes is asymptotically equal to its fractional variant. We introduce this variant, the `fractional extremal function' in ~\cref{s:fracintro} and summarize its properties. It turns out that the fractional extremal function is much better behaved. For example, we are able to determine it exactly for all graphs $H$ with $\alpha(H) \leq \frac{\v(H)}{3}$.
 
The bulk of the paper, namely Sections~\ref{s:frac}--\ref{s:upper}, is devoted to investigation of the fractional  extremal function, including the proof of its asymptotic equivalence to the ordinary extremal function, mentioned above, and proofs of sharper analogues of Theorems~\ref{t:main0}--\ref{t:upper0} for the fractional extremal function. A more detailed outline of the organization of these sections is given in \cref{s:fracintro}.

Finally, in \cref{s:general} we turn our attention to the extremal function of regular graphs with superconstant density. In particular, we show that the extremal function of hypercubes is linear in the number of vertices, while the extremal function of any regular graph with density superlogarithmic in the number of vertices is superlinear  in the number of vertices.

\subsection{Notation}

Let $\NN:=\{1,2,\dots\}$ and $\NN_0:=\{0,1,\dots\}$. For $n,m\in\NN_0$, let $\mathdefn{[m,n]} := \{m,m+1,\dots, n\}$ and $\mathdefn{[n]} := \{1,2,\dots, n\}$. Let $\mathbb{R}$ be the set of real numbers, and let $\mathbb{R}_+$ be the set of non-negative real numbers.% Let $\mathbb{Z}_+$ be the set of non-negative integers. \comment{We should use $\NN_0$ xor $\mathbb{Z}_+$.}

We denote by \defn{$\delta(G)$} the minimum degree of a graph $G$, and by \defn{$\kappa(G)$}  the vertex connectivity of $G$. For $X \subseteq V(G)$, let \defn{$G[X]$} be the subgraph of $G$ induced by $X$.

A \defn{$k$-blowup $G^{(k)}$} of a graph $G$ is obtained from $G$ by replacing every vertex $v$ of $G$ with an independent set $I(v)$ of size $k$ and for every edge $uv \in E(G)$ joining every vertex of $I(u)$ to every vertex of $I(v)$.
 
A \defn{linkage} in a graph $G$ is a collection of pairwise vertex-disjoint paths in $G$. A linkage $(P_1,\dots,P_k)$ is an \defn{$(s_1t_1,\dots, s_kt_k)$-linkage} if $P_i$ has ends $s_i$ and $t_i$ for each $i \in [k]$.
 
We denote the components of a vector $\bl{w} \in \bb{R}^d$ by $w_1,w_2,\dots,w_d$. For finite $X \subseteq \NN$, let \defn{$\bl{1}_{X}$} be the characteristic vector of $X$; that is, the $i$-th component of $\bl{1}_{X}$ is equal to $1$ if $i \in X$, and is equal to $0$ otherwise.

%%%%%%%%%%%%%%%%%%%%%%
\section{The fractional extremal function}\label{s:frac}

\subsection{Introduction}\label{s:fracintro}

In this section we discuss the properties of fractional extremal functions and state our main technical results, which immediately imply \cref{t:main0} and \cref{t:upper0}. The formal definition of the fractional extremal function $c_f(H)$ is rather technical and we give it in the next subsection. However, as we eventually prove at the end of \cref{s:connect}, it can be alternatively defined as a natural scaling limit\footnote{Analogously, for example, the fractional chromatic number of a graph $H$ (see \citep{SU97}) can be defined as the scaling limit, 
$ \chi_f(H) = \lim_{k \to \infty} \frac{\chi(H \boxtimes K_k)}{k}$, where $H\boxtimes K_k$ is the strong product of $H$ and $K_k$ (the graph obtained from $H$ by replacing each vertex $v$ of $H$ by a $k$-clique $C_v$ and replacing each edge $vw$ of $H$ by a complete bipartite graph between $C_v$ and $C_w$).
} of the usual extremal function; that is, 
\begin{equation}\label{e:frac0}
c_f(H) = \lim_{k \to \infty}\frac{c(k\,H)}{k}. 
\end{equation}
We use \eqref{e:frac0} as the definition of the fractional extremal function for the purposes of this introduction. Note that the extremal function $c(k\,H)$ has been investigated previously, for example when $H$ is complete \cite{Tho08} or when $H$ is a cycle \cite{HarWoo15,CLNWY}.

The fractional extremal function has several desirable properties, which makes it substantially easier to work with in comparison to the original:
\begin{itemize}
	\item \cref{lem:boundedDiff} shows that $c_f(H) \leq c_f(H \setminus v)+1$ for every graph $H$ and every $v \in V(H)$. This ``continuity'' implies in particular that $c_f(H) \leq \v(H)-1$ for every graph $H$, with equality for complete graphs.
	\item For a graph $H$, let $$\mathdefn{\gamma_H} := \sup\{\, \d(G) \: | \:  G \;\text{is complete multipartite}, H \not\preceq G\}.$$
	\cref{lem:cf3} and \cref{l:newlower} show that $\gamma_H\leq c_f(H)\leq \gamma_H+2$. 	 
	Thus it (mostly) suffices to consider only complete multipartite graphs in our investigation of $c_f(H)$. As $\gamma_H \leq c(H)$, it also follows that ${c_f(H)\leq c(H)+2}$.
\end{itemize}

As noted above, $c_f(K_t) = t-1$ for every $t \geq 2$. Thus by \cref{e:Thomason}, the fractional and integral extremal function of the complete graphs differ substantially. However, as the next theorem shows the difference becomes negligible as we consider large graphs in s.s.s.\ families.

\begin{restatable}{thm}{Main}\label{t:main1} 
For every s.s.s.\ family $\mc{F}$ and for every $H \in \mc{F}$,
$$c(H) = (1+o(1))c_f(H).$$
\end{restatable}

By \cref{t:main1} it suffices to prove Theorems~\ref{t:main0} and~\ref{t:upper0} for the fractional extremal function. We show that \cref{e:main0} holds for the fractional extremal function of 4-colourable graphs.

\begin{restatable}{thm}{Fourcolored}\label{t:4colored} 
For every 4-colorable graph $H$, 
	\begin{equation}\label{e:3colored}
	c_f(H) =  \max \s{\frac{\v(H)}{2},\tau(H)}.
	\end{equation}
\end{restatable}

The arguments presented in \cref{s:intro} immediately imply that $c_f(H) \geq \max \s{\frac{\v(H)}{2},\tau(H)}$ for every graph $H$.  The other direction does not always hold, as the following example shows, with the help of \cref{t:main1}.

\begin{example}
	\label{InterestingExample}
Let $H_t$ be the graph obtained from $K_{2t}$ by deleting the edges of some complete subgraph on $t$ vertices. Let $T(4n,4)$ be the balanced complete $4$-partite graph $K_{n,n,n,n}$. Every subgraph of $T(4n,4)$ that has a $K_{t+1}$ minor has at least $2t-2$ vertices, since all but four vertices of  the minor must be obtained by contraction. A similar argument shows that every subgraph of $T(4n,4)$ that has an $H_t$ minor has at least $3t-3$ vertices. 

Suppose that $n = \lfloor\frac{k(3t-3)-1}{4}\rfloor$ for some  $k \in \NN$. Then $k\,H_t$ is not a minor of $T(4n,4)$, and thus $$c(k\,H_t) \geq \d(T(4n,4)) = \frac{3}{2}n \geq \frac{3}{8}k(3t-3)-1.$$ 
It follows that $c_f(H_t) \geq \frac{9}{8}(t-1)$. Meanwhile, $\tau(H_t)= \frac{\v(H_t)}{2} =t$, and so  for $t \geq  10$, 
\begin{equation} \label{e:ht}
c_f(H_t) \geq \frac{81}{80}\max \s{\frac{\v(H_t)}{2},\tau(H_t)}.
\end{equation} 
Note that $H_t$ is $(t+1)$-colourable.   

Now assume $t \geq 10$, and let $\mc{F}_t$ be the family consisting of all graphs $G$ isomorphic to a subgraph of $k\,H_t$ for some $k\in\NN$. Clearly, $\mc{F}_t$ is an  s.s.s.\ family. By \eqref{e:ht},  $$c_f( k\,H_t)  = k c_f(H_t) \geq  k\s{\frac{81}{80}\max \s{\frac{\v(k H_t)}{2},\tau(k H_t)}}=\frac{81}{80} \max \s{\frac{\v(k\,H_t)}{2},\tau(k\,H_t)}.$$ 
It now follows from \cref{t:main1} that the sequence of graphs $\{k\,H_t\}_{k \in \bb{N}}$ does not satisfy \eqref{e:main0}. 
Thus the condition $\chi(H) \leq 4$ in \cref{t:main0} can not be relaxed  to $\chi(H) \leq 11$, as claimed in the introduction.
\end{example}	

A similar proof shows the following:

\begin{lem}
For every graph $H$ and $i \in \bb{N}$,
$$c_f(H) \geq \frac{i-1}{i}\s{\v(H) - \frac{\alpha_i(H)}{2}}.$$
\end{lem}

\begin{proof} Generalizing the notation in \cref{InterestingExample}, let $T(in,i)$ denote the graph with the vertex set admitting a partition $(A_1,\dots,A_i)$ such that $|A_j|=n$ for every $j \in [i]$, and two vertices of $T(in,i)$ are adjacent if and only if they belong to different parts of this partition.
	
Let $G$ be a subgraph of $T(in,i)$ such that $H$ is a minor of $G$. Then all but at most $\alpha_i(H)$ vertices of $H$ must be obtained by contraction, implying
$$\v(G) \geq \alpha_i(H) + 2 ( \v(H) - \alpha_i(H) ) = 2\v(H) - \alpha_i(H).$$
Thus, if $k\,H$ is a minor of $T(in,i)$ then $in \geq k ( 2\v(H) - \alpha_i(H) )$. 
Let $n = \ceil{  \frac{k}{i} ( 2\v(H) - \alpha_i(H) ) }-1$.
Thus	$in < k ( 2\v(H) - \alpha_i(H) )$, implying $k\,H$ is not a minor of $T(in,i)$. 
Hence
	$$c(k\, H) \geq \d( T(in,i) ) = \frac{\binom{i}{2}N^2}{iN} = 
	\frac{i-1}{2} \left( \ceil*{  \frac{k}{i} ( 2\v(H) - \alpha_i(H) ) }-1 \right)$$
	and
$$c_f(H) = \lim_{k\to\infty} \frac{c(k\, H)}{k}  \geq 
 \frac{i-1}{i} \left(  \v(H) - \frac{\alpha_i(H)}{2} \right),$$	as desired.
\end{proof}

This lemma motivates us to define 
$$\mathdefn{c_T(H)} := \sup\left\{ \tau(H), \frac{2}{3}\s{\v(H) - \frac{\alpha_3(H)}{2}},\dots, \frac{i-1}{i}\s{\v(H) - \frac{\alpha_i(H)}{2}},\dots\right\}.$$
It is tempting to conjecture that $c_f(H)=c_T(H)$ for every graph $H$; that is, the  balanced complete multipartite (Tur\'an) graphs  are the only source of extremal examples. We are unable to prove this in general, but have proved it for large enough $c_f(H)$. 

\begin{restatable}{thm}{Twothirds}\label{t:twothirds}
	For every graph $H$ such that $c_f(H) > \frac{2}{3}\v(H)$, 
	$$c_f(H)=c_T(H).$$
\end{restatable}

\subsection{Formal definition and basic properties}

In this  subsection, we  formally define the fractional extremal function and establish a number of its basic properties.

Let $\mc{P}(X)$ denote the collection of all subsets of a set $X$. A \defn{model} of a graph $H$ in a graph $G$ is a function $\mu: V(H) \to \mc{P}(V(G))$ such that:
\begin{itemize} 
	\item[(M1)] $G[\mu(v)]$ is connected for every $v \in V(H)$,
	\item[(M2)] for every edge $uv \in E(H)$ there exists $u' \in \mu(u)$ and $v' \in \mu (v)$ such that 
	$u'v' \in E(G)$,
	\item[(M3)]  $\mu(v) \cap \mu(u) = \emptyset$ for every pair of distinct $u,v \in V(H)$.
\end{itemize}
Observe that $H$ is a minor of $G$ if and only if there exists a model of $H$ in $G$

We need the following weakening of the notion of a model. A \defn{jumbled model}\footnote{A different notion of fractional model restricted to complete graphs, called \defn{fractional brambles}, was independently introduced by \citet{Fox11} and \citet{Pedersen11}.} of a graph $H$ in a graph $G$ is a function $\mu: V(H) \to \mc{P}(V(G))$ satisfying the properties (M1) and (M2) above, but not necessarily (M3). That is, we allow the images of vertices of $H$ in $G$ to overlap. We keep track of the overlaps by defining $\mu^{\#}(u)$, for $u \in V(G)$, to be the number of vertices $v \in V(H)$ such that $u \in \mu(v)$.

Define the \defn{$H$-volume} of a graph $G$, denoted by $\mathdefn{\brm{Vol}_{H}(G)}$, to be the maximum weight of a fractional packing of jumbled models of $H$. That is, 
$$\brm{Vol}_{H}(G) := \sup \sum_{i=1}^n \alpha_i,$$
taken over all choices of $n \in \bb{N}$, $\alpha_1,\dots,\alpha_n \in \bb{R}_+$ and jumbled models $\mu_1,\dots,\mu_n$ of $H$ in $G$ such that for every $v \in V(G)$, 
\begin{equation}
\label{e:Vol} \sum_{i=1}^n \alpha_i\mu_i^{\#}(v) \leq 1.
\end{equation}
It is not hard to see that the supremum in the definition of $H$-volume is always achieved, since it is a maximum of a  linear function on a bounded polytope in $\bb{R}_+^{\v(G)}$.

If $H \preceq G$ then there exists a model $\mu$ of $H$ in $G$. Taking $k=1$, $\alpha_1=1$ and $\mu_1 = \mu$, we see that $\brm{Vol}_{H}(G) \geq 1$. The converse  does not hold in general.

Similarly, if $\ell H \preceq G$ then $\brm{Vol}_{H}(G) \geq \ell$. The next lemma shows that if $G$ is a large blowup of a bounded size graph, then  a partial converse of this statement holds.

\begin{lem}\label{lem:multipartite1}\label{m1}
	For all $T  > 0$ there exists $K=K_{\ref{m1}}(T)>0$ such that for all graphs $H$ and $G$  such that $\v(H),\v(G) \leq T$ and for all $k,\ell \in\NN_0$ such that $$\ell \leq k\brm{Vol}_H(G) - K,$$
	there exists a model $\mu$ of $\ell H$ in $G^{(k)}$ such that $|\mu(v)| \leq \v(G)$ for every $v \in V(\ell H) $. In particular, $\ell H$ is a minor of $G^{(k)}$.
\end{lem}

\begin{proof}
	We show that $K=T^{2^{T}}$ satisfies the lemma.
	By definition of the $H$-volume there exist  $n\in \bb{N}$, $\alpha_1,\dots,\alpha_n \in \bb{R}_+$ and jumbled models $\mu_1,\dots,\mu_n$ of $H$ in $G$ such that $k \sum_{i=1}^n \alpha_i \geq \ell+K$ and for every $v \in V(G)$, 
	\begin{equation*}
	\sum_{i=1}^n \alpha_i\mu_i^{\#}(v) \leq 1,
	\end{equation*}
	For $v \in V(H)$, let $I(v)$ be the set of $k$ vertices of  $G^{(k)}$corresponding to $v$, as in the definition of the blowup.
	
	Note that for each $i$, and any $Z \subseteq V(G^{(k)})$ such that $|Z \cap I(v)| = \mu_i^{\#}(v)$ there exists a model $\mu'_i$ of $H$ in $G^{(k)}$ naturally corresponding to the jumbled model $\mu_i$, and  $|\mu'_i(v)| \leq  \v(G)$ for every $v \in V(H)$.
	
Since there are at most ${\v(H)}^{2^{\v(G)}}$ distinct jumbled models of $H$ in $G$ we assume $n \leq K$. Let $\beta_i = \lfloor k \alpha_i \rfloor$ for $i \in [n]$. By the above, we can find a model of $\beta_i H$  using   $\beta_i\mu_i^{\#}(v)$ vertices in $I(v)$,
	and we can choose these models to be disjoint for different $i$.
	Now since
	$$\sum_{i=1}^{n}\beta_i \geq  \s{k \sum_{i=1}^n \alpha_i} - n \geq k\brm{Vol}_H(G) - K\geq \ell,$$
	the union of the above models is a model of $\ell' H$  for some $\ell'\geq \ell$, and so $\ell H$ is a minor of $G^{(k)}$. 
\end{proof}

Secondly, using LP duality we establish a lower bound on $K_{s,t}$-volume. 

\begin{lem}\label{lem:bipartite}
	For all integers $s \geq t > 0$ and every graph $G$, 
	$$\brm{Vol}_{K_{s,t}}(G) \geq \min\left\{ \frac{\v(G)}{s+t}, \frac{\delta(G)}{t} \right\}. $$
\end{lem}

\begin{proof} 
	Let $\bar{E}$ denote the set of ordered pairs  of adjacent vertices of $G$.
	For $(u,v) \in \bar{E}$   define $\mu_{uv}$ to be the jumbled model of $K_{s,t}$ which maps the  part of the bipartition of $K_{s,t}$ of size $s$ to $\{u\}$, and the other part to $\{v\}$. Thus $\mu^{\#}_{uv}(u)=s$, $\mu^{\#}_{uv}(v)=t$, and $\mu^{\#}_{uv}(x)=0$  for every $x \in V(G) -\{u,v\}$. It follows from the definition of $K_{s,t}$-volume that $\Vol_{K_{s,t}}(G)$ is at least the maximum of $\sum_{(u,v) \in \bar{E}}\alpha_{u,v}$ taken over the sequences of  non-negative real numbers $(\alpha_{uv})_{(u,v) \in \bar{E}}$ such that 
	$$ \sum_{v: uv \in E(G)}(s\alpha_{uv}+t\alpha_{vu}) \leq 1,$$
	for every $u \in V(G)$. By LP duality the above maximum is equal to the minimum of  $\sum_{v \in V(G)}\beta(v)$ over the functions $\beta: V(G) \to \bb{R}_{+}$ such that
	$s\beta(u)+t\beta(v) \geq 1$ for all $uv \in E(G)$. Thus it suffices to show that, for all such functions  
	$$\sum_{v \in V(G)}\beta(v) \geq \min \left\{ \frac{\v(G)}{s+t}, \frac{\delta(G)}{t} \right\}.$$
	Let $u \in V(G)$ be a vertex with $\beta(u)$ minimum and let $\beta_0 = \beta(u)$. If $\beta_0 \geq \frac{1}{s+t}$  then $\sum_{v \in V(G)} \beta(v) \geq \beta_0 \v(G) \geq \frac{\v(G) }{s+t}$, as desired. Otherwise, $\beta(v) \geq \frac{1-s\beta_0}{t}$ for each neighbour $v$ of $u$. It follows that
	\begin{align*}
	\sum_{v \in V(G)} \beta(v)\geq \frac{1 - s\beta_0}{t} \delta(G)  + \beta_0(\v(G)-\delta(G)) \geq \min \left\{ \frac{\v(G)}{s+t}, \frac{\delta(G)}{t} \right\}, 
	\end{align*}  
	where the second inequality follows as the linear function of $\beta_0$ on the interval $[0,\frac{1}{s+t}]$ achieves its minimum for $\beta_0=0$ or $\beta_0= \frac{1}{s+t}$.
\end{proof}

Finally, define the \defn{fractional extremal function} of a graph $H$ to be
\begin{equation}\label{e:cf1} 
c_f(H) := \sup\frac{\d(G)}{\brm{Vol}_{H}(G)} 
\end{equation} 
taken over all non-null graphs $G$. 

We finish this section by proving a crucial lemma that in particular implies that $c(k\,H) \leq (1+o(1))k\,c_f(H)$ for every $H$. Its proof relies on the following standard corollary of Szemer\'{e}di's regularity lemma combined with an embedding lemma, which follows immediately from the degree form of the regularity lemma~\cite[Theorem 1.10]{KomSim93} and a version of the embedding lemma given in ~\cite[Theorem 2.1]{KomSim93} and a follow-up remark. The famous blowup lemma of Koml\'{o}s, S\'{a}rk\"{o}zy and Szemer\'{e}di~\cite{KSS97Blowup} implies that the restriction on the component size in \cref{emb} can be replaced by the much weaker restriction that the maximum degree of $G'$ is at most $K$. The current version, however, is sufficient for our purposes and is fairly straightforward to prove. For completeness we include a proof in Appendix~A.

\begin{restatable}{thm}{Emb}\label{t:embedding}\label{emb} 
	For all $\eps > 0$ there exists $T=T_{\ref{emb}}(\eps)$ such that for every $K$ there exists $N=N_{\ref{emb}}(\eps,K)$ satisfying the following. For every graph $G$ with $\v(G) \geq N$, there exists a graph $R$ with $\v(R) \leq T$ and $k\in\NN$ such that: 
	\begin{itemize}
		\item $(1-\eps)\,\v(G) \leq k\v(R) \leq \v(G)$,
		\item $k \delta(R) \geq \delta(G) - \eps\,\v(G)$,
		\item $k^2\e(R) \geq \e(G) - \eps\,\v^2(G)$,
		\item if  $G'$ is a subgraph of $R^{(k)}$ such that every component of $G'$ has at most $K$ vertices, then $G'$ is isomorphic to a subgraph of $G$.
		\end{itemize}	
\end{restatable}

Combining \cref{lem:bipartite} with \cref{t:embedding} we obtain the following.

 \begin{lem}\label{l:copies}
 	For all $\eps > 0$ and every graph $H$ there exists $L=L_{\ref{l:copies}}(\eps,H)$ satisfying the following. Let $G$ be a graph and let  $\ell \geq L$ be an integer such that at least one of the following conditions holds:
 	\begin{enumerate}
 		\item[(i)] $\d(G) \geq \ell c_f(H) +\eps\,\v(G)$, or
 		\item[(ii)] $\v(G) \geq 2(1+\eps) \ell \v(H)$ and $\delta(G) \geq \ell \tau(H)+\eps\,\v(G)$. 
 	\end{enumerate}
 	Then $\ell\, H$ is a minor of $G$.
 \end{lem}
 
 \begin{proof} If $H$ is the null graph, then the lemma is trivial. Now assume $H$ is non-null, and so $c_f(H) \geq\frac12$. 
 	Let $t:=\v(H)$, let $T :=T_{\ref{emb}}(\frac{\eps}{2})$, $K := tT$, $K':=K_{\ref{m1}}(\max\{T,t\})$, and $N:=N_{\ref{emb}}(\frac{\eps}{2},K)$. We now show that $L:=\max\{2N, 4\eps^{-1}K'\v(H) \}$ satisfies the lemma. Both (i) and (ii) imply that $\v(G) \geq L/2\geq N$. Thus by \cref{t:embedding} there exists a graph $R$ with $\v(R)\leq T$ and $k\in\NN$ satisfying the conditions of \cref{t:embedding}.
 	In particular,
 	\begin{itemize}
 	\item $k\d(R)  \geq \d(G)- \frac{\eps}{2}\v(G)$,
 	\item $k\delta(R) \geq \delta(G) - \frac{\eps}{2}\v(G)$,
 	\item $k\v(R) \geq (1-\frac{\eps}{2})\,\v(G)$.
 	\item if   $G'$ is a subgraph of $R^{(k)}$ such that every component of $G'$ has at most $K$ vertices, then $G'$ is isomorphic to a subgraph of $G$.
 \end{itemize}
 	If (i) holds then since $\v(G)\geq \d(G)$, 
 	$$ k\Vol_H(R) \geq \frac{k\d(R)}{c_f(H)} \geq \frac{1}{c_f(H)}\s{\d(G)-\frac{\eps}{2}\v(G)} \geq  \ell + \frac{\eps\,\v(G)}{2c_f(H)} \geq \ell+K'.$$
 	If (ii) holds then   $$k\Vol_H(R) \geq k\Vol_{K_{\v(H),\tau(H)}}(R) \geq \min\left\{ \frac{k\v(R)}{2\v(H)}, \frac{k\delta(R)}{\tau(H)} \right\}, $$ where the first inequality holds since $H$ is a minor of $K_{\v(H),\tau(H)}$, and the second inequality holds by \cref{lem:bipartite}. Since $k\v(R) \geq 2\ell \v(H) + \frac{\eps}{2}\v(G) $ and $k\delta(R) \geq \ell \tau(H) + \frac{\eps}{2}$ by (ii) and the properties of $R$ listed above,
 	$$\min\left\{ \frac{k\v(R)}{2\v(H)}, \frac{k\delta(R)}{\tau(H)} \right\} \geq \ell + \frac{\eps\,\v(G)}{4\v(H)} \geq \ell+K'.$$
 	 	
 	Thus, in both cases $k\Vol_H(R) \geq \ell+K'$. By \cref{m1} there exists a model $\mu$ of $\ell H$ in $R^{(k)}$ such that $|\mu(v)| \leq T$ for every $v \in V(\ell H)$. It follows that there exists a subgraph of $G'$ of  $R^{(k)}$  such that every component of $G'$ has at most $K$ vertices and  $\ell H$ is a minor of $G'$. By the choice of $R$ it follows that $G'$ is a subgraph of $G$, and so $\ell H$  is a minor of $G$ as desired.
 \end{proof}

\section{From  graphs to weight vectors}\label{s:weight}

It is convenient for us to further the fractional point of view and extend the definitions of density and $H$-volume to  weighted graphs. For our purposes, a \defn{weighted graph} is a pair $(G,w)$, where $w:V(G)\to \bb{R}_+$ is a function such that $|w|= \sum_{v \in V(G)} w(v) >0$. Let 
$$\e(G,w) := \sum_{uv\in E(G)} w(u) w(v) \quad \text{and} 
\quad \d(G,w) := \frac{\e(G,w)}{|w|}.$$

As a natural extension of the definition of  $\brm{Vol}_{H}(G)$, define the  \defn{$H$-volume $\brm{Vol}_{H}(G,w)$} of  $(G,w)$ by replacing condition \eqref{e:Vol} by
\begin{equation}
\label{e:wVol}\sum_{i=1}^n \alpha_i\mu_i^{\#}(v) \leq w(v),
\end{equation}
for every $v \in V(G)$. Note that $H$-volume scales linearly with the weights; that is, 
$$\brm{Vol}_{H}(G,\alpha \cdot w) = \alpha\brm{Vol}_{H}(G,w) \qquad \text{for every} \; \alpha >0.$$

We can naturally define the sum of weighted graphs $(G_1,w_1)$ and $(G_2,w_2)$ by 
$$ (G_1,w_1)+ (G_2,w_2) := (G_1 \cup G_2,w_1+w_2),$$ 
using the convention $w_i(v)=0$ for $i \in [2]$ and $v \in V(G_{3-i})-V(G_i)$. 
The $H$-volume is clearly superadditive with respect to this operation; that is, 
\begin{equation}\label{e:volumeAdd}
\brm{Vol}_{H}((G_1,w_1)+ (G_2,w_2)) \geq \brm{Vol}_{H}(G_1,w_1) + \brm{Vol}_{H}(G_2,w_2).
\end{equation}

In the special case that $w:V(G)\to \NN_0$, the weighted graph $(G,w)$ naturally corresponds to an unbalanced blowup of $G$, which we denote by $G^w$ obtained from $G$ by replacing every vertex $v$ of $G$ with an independent set $I(v)$ of size $w(v)$ and for every edge $uv \in E(G)$ joining every vertex of $I(u)$ to every vertex of $I(v)$. Note that a jumbled model $\mu$ of a graph $H$ in $G^{w}$ corresponds to a jumbled model $\mu_{*}$ of $H$ in $G$ such that $\mu_{*}^{\#}(v) = \sum_{u \in I(v)}\mu^{\#}(u)$ for every $v \in V(G)$. It follows that $\Vol_H(G,w) = \Vol_H(G^{w}),$ when $w$ is integral.
From this it can easily be shown that $$\Vol_H(G,w) = \lim_{k \to \infty} \frac{1}{k}\Vol_H(G^{\lfloor k \cdot w \rfloor}). $$

Since $\d(G,w) = \lim_{k \to \infty} \frac{1}{k}\d(G^{\lfloor k \cdot w \rfloor})$, it follows that
$$\frac{\d(G,w)}{\Vol_H(G,w)} = \lim_{k \to \infty} \frac{\d(G^{\lfloor k \cdot w \rfloor})}{\Vol_H(G^{\lfloor k \cdot w \rfloor})},$$
implying that the fractional extremal function can be equivalently defined by 
\begin{equation}\label{e:cf11} 
c_f(H) = \sup\frac{\d(G,w)}{\brm{Vol}_{H}(G,w)} 
\end{equation} 
taken over all weighted graphs $(G,w)$. Since both the density and volume are linear in the weight function we can equivalently define 
\begin{equation}\label{e:cf2} 
c_f(H) = \sup_{(G,w) \: : \: \brm{Vol}_{H}(G,w) < 1} \d(G,w).
\end{equation}
The second interpretation makes the connection to the non-fractional  extremal function perhaps more evident. A definition of $c(H)$ can be obtained from \eqref{e:cf2} by replacing weighted graphs with unweighted ones, and the condition $\brm{Vol}_{H}(G,w) < 1$ with the condition that $G$ has no $H$ minor.

Our next goal is to  simplify the  definition of the fractional extremal function by showing that it suffices to take the supremum in \eqref{e:cf11} over complete weighted graphs. Towards this goal, the following lemma shows that the density of a weighted graph can be split over its complete subgraphs.  

\begin{lem}\label{lem:weightedCliques1}
	For every weighted graph $(G,w)$ there exist complete weighted graphs $(G_1,w_1),\dots,(G_m,w_m)$ such that 
	$$(G,w) = \sum_{i=1}^{m} (G_m,w_m) \qquad and \qquad \d(G,w) =\sum_{i=1}^m \d(G_i,w_i).$$ 
\end{lem}

\begin{proof}
	We prove the lemma by induction on $\v(G)$. The base case $\v(G)=1$ is trivial.
	
	For the induction step, note that if $G$ is complete, then the lemma trivially holds. Further, if $w(v)=0$ for some $v \in V(G)$, then the lemma follows by applying the induction hypothesis to the graph $G \setminus v$.
	
Thus we assume that there exist non-adjacent $a,b \in V(G)$, such that $w_0 := w (a)+w(b)>0$. Let $\alpha=\frac{w(a)}{w_0}$. Let $G^a = G \setminus b$, and let  $w^a: V(G^a) \to \bb{R}$ be defined by $w^a(a)=w(a)$ and $w^a(v)=\alpha w(v)$ for every $v \in V(G) - \{a,b\}$. Similarly, define $G^b = G \setminus a$, and  define $w^b: V(G^b) \to \bb{R} $ by 	$w^b(b)=w(b)$ and $w^b(v)=(1 -\alpha) w(v)$ for every $v \in V(G) - \{a,b\}$. Note that 
	\begin{equation} 
	\label{e:clique1} (G,w)=(G^a,w^a) + (G^b,w^b).
	\end{equation}
	
	Let $f=\sum_{e \in E(G \setminus \{a,b\}), e=uv} w(u)w(v)$, let $d_a = \sum_{u : au \in E(G)}w(a)w(u)$ and let $d_b = \sum_{u : au \in E(G)}w(a)w(u)$. Then $\d(G,w)=(f+d_a+d_b)/w(V(G))$, $\d(G^a,w^a)=(\alpha^2 f+\alpha d_a)/(\alpha w(V(G)))$ and $\d(G^b,w^b)=((1-\alpha)^2 f+(1-\alpha) d_b)/((1-\alpha) w(V(G)))$.
	Thus 
	\begin{equation} 
	\label{e:clique2}\d(G,w)=\d(G^a,w^a)+\d(G^b,w^b).
	\end{equation}
	The lemma immediately follows from \eqref{e:clique1}, \eqref{e:clique2} and the induction hypothesis applied to $(G^a,w^a)$ and $(G^b,w^b)$.
\end{proof}

\begin{lem}\label{lem:cf3} 
For every graph $H$, 
	 \begin{equation}\label{e:cf3} 
	 c_f(H) = \sup_{(G,w)\: :\: G \; \mathrm{is \; complete} }\frac{\d(G,w)}{\brm{Vol}_{H}(G,w)}. 
	 \end{equation}
\end{lem}

\begin{proof}
	Let $c$ be the supremum on the right side of \eqref{e:cf3}. It suffices to show that $\d(G,w) \leq c \brm{Vol}_{H}(G,w)$ for every weighted graph $(G,w)$. Let $(G_1,w_1),\dots,(G_m,w_m)$ be the complete weighted graphs satisfying the conditions of \cref{lem:weightedCliques1}. Then
	\begin{align*}
	\d(G,w)  = \sum_{i=1}^m \d(G_i,w_i) \leq c \sum_{i=1}^m \brm{Vol}_{H}(G_i,w_i) \stackrel{\eqref{e:volumeAdd}}{\leq}  c\brm{Vol}_{H}(G,w),
	\end{align*}
	as desired.  
\end{proof}	

 \cref{lem:cf3} allows us to dispense with the graph structure of $G$ in the definition of the fractional extremal function and work in the linear algebraic setting with the vector of weights. Note that in any minimal jumbled model $\mu$ of $H$ in a complete graph, $|\mu(v)| \in \{1,2\}$ for every $v \in V(H)$. Thus we define a \defn{jumbled  model} of $H$ in $\bb{N}$ to be a map $\mu: V(H) \to \mc{P}(\bb{N})$ such that:
 \begin{itemize} 
 	\item[(N1)] $|\mu(v)| \in \{1,2\}$ for every $v \in V(H)$, and
 	\item[(N2)] If $\mu(v)=\mu(u)  =\{i\}$ for some $i \in \bb{N}$ and $u,v \in V(G)$ then $u$ and $v$ are non-adjacent.
 \end{itemize}
 We define a vector $\bl{\mu}^{\#}=(\mu^{\#}_i)_{i \in \bb{N}}$ corresponding to $\mu$ as before; that is, $\mu^{\#}_i$ is equal to the number of vertices $v \in V(H)$ such that $i \in \mu(v)$. Let $\#(H)$ denote the convex hull of  the vectors $\bl{\mu}^{\#}$  of all jumbled  models $\mu$ of $H$ in $\bb{N}$. 
 Given $X \subseteq \bb{R}$ let
$X^{\bb{N}}$ (non-standardly) denote the space of vectors $\bl{v}=(v_i)_{i \in \bb{N}}$ with finite support.\footnote{We often work with $\bb{R}_+^{\bb{N}}$ rather then a finite dimensional space to avoid repeatedly specifying the dimension.} For a vector $\bl{w}\in \bb{R}_{+}^{\bb{N}}$ define the a \defn{$H$-volume} of $\bl{w}$, denoted by $\Vol_H(\bl{w})$, as the supremum of $\alpha \geq 0$ such that there exists $\bl{x} \in \#(H)$  such that $\alpha \bl{x} \leq \bl{w}$.  Note that if $G$ is the complete graph with the (finite) vertex set $\brm{supp}(\bl{w})$  then $$\Vol_H(G,\bl{w}|_{V(G)}) = \Vol_H(\bl{w}).$$
Note also that it follows from \eqref{e:volumeAdd} that $\Vol_H$ is superadditive.

Next, to convert the notion of density, define a (non-standard) inner product on $\bb{R}_+^{\bb{N}}$  of vectors $\bl{w}=(w_i)_{i \in \bb{N}}$ and $\bl{w}'=(w'_i)_{i \in \bb{N}}$ by $$\ip{\bl{w}}{\bl{w}'} = \sum_{(i,j) \in \bb{N}^2, i\neq j} w_iw'_j.$$

Let $|\bl{w}| =\sum_{i \in \bb{N}}  w_i$ and 
$\|  \bl{w}\|=\ip{\bl{w}}{\bl{w}}$. 
For  $\bl{w} \in \bb{R}_+^{\bb{N}} \setminus \{\bl{0}\}$, let  $$\d(\bl{w})= \frac{\| \bl{w}\|}{2|\bl{w}|}.$$

We identify $\bb{R}_+^d$ with the set of vectors in $\bb{R}_+^{\bb{N}}$ with support in  $[d]$, which in particular allows us to use the above notation for vectors in $\bb{R}_+^d$.  

Restating \cref{lem:cf3} in our new setting we obtain that for every graph $H$,
\begin{equation}\label{e:cf4} 
c_f(H) = \sup_{\bl{w} \in \bb{R}_+^{\bb{N}} \setminus \{\bl{0}\}}\frac{\d(\bl{w})}{\brm{Vol}_{H}(\bl{w})}. 
\end{equation}

\begin{lem}\label{cfLowerBound}
	For every graph $H$, 
	$$c_f(H) \geq c_T(H).$$
\end{lem}
 
\begin{proof}Let $k \geq 2$ be an integer, and let $\mu$ be a jumbled model of $H$ in $\bb{N}$ such that $\mu(v) \subseteq [k]$ for every $v \in V(H)$. Let $S = \{v \in V(H)\: | \: |\mu(v)|=1\}$. Then $\{v \in S | \mu(v) = \{i \} \}$ is independent for every $i \in [k]$, and so $|S| \leq \alpha_k(H)$. Thus $|\bl{\mu}^{\#}| \geq |S| + 2 (\v(H)-|S|) \geq 2\v(H) - \alpha_k(H)$. It follows that $|\bl{x}|_1 \geq  2\v(H) - \alpha_k(H)$ for every $\bl{x} \in \#(H)$ such that $\brm{supp}(\bl{x}) \subseteq [k]$. Consequently, for every $\bl{w} \in \bb{R}_{+}^{\bb{N}}$ such that $\brm{supp}(\bl{w}) \subseteq [k]$, 
	$$\Vol_H(\bl{w}) \leq \frac{|\bl{w}|}{2\v(H) - \alpha_k(H)}.$$
	
Let $\bl{w} \in \bb{R}_{+}^{\bb{N}}$ be defined by $w_i=1$ for $i \leq k$ and $w_i=0$ for $i>k$. Then $\d(\bl{w}) = \frac{k-1}{2}$, and so
$$c_f(H) \geq \frac{\d(\bl{w})}{\brm{Vol}_{H}(\bl{w})} \geq \frac{(k-1)}{2} \frac{2\v(H) - \alpha_k(H)}{k} = \frac{k-1}{k}\s{\v(H)-\frac{\alpha_{k}(H)}{2}}. $$

It remains to show that $c_f(H) \geq \v(H) - \alpha(H).$ Let  $\bl{w}^{r}=(r,1,0,\dots,0,\dots)$ for $r \in \bb{R}_{+}$. By definition $\Vol_{H}(\bl{w}^{r})$ is the supremum of $\gamma > 0$ such that there exists a convex combination $\bl{x} \in \bb{R}_+^{\bb{N}}$ of vectors $\bl{\mu}^{\#}$ of jumbled models of $H$ in $\bb{N}$ such that $\gamma\bl{x} \leq \bl{w}^{r}$. As $\bl{w}^{r}_i=0$ for $i \geq 3$ it suffices to consider jumbled models $\bl{\mu}$ of $H$ such that $\mu(v) \subseteq \{1,2\}$  for every $v \in V(H)$. As noted before, $\{v \in V(H) \: | \: \mu(v) = \{1 \}  \}$ is independent. Therefore, $\mu^{\#}_2 \geq  \v(H) - \alpha(H)$ for every such $\mu$, and so   $x_2 \geq  \v(H) - \alpha(H)$ for every $\bl{x} \in \#(H)$ such that $\brm{supp}(\bl{x}) \subseteq  \{1,2\}$. 

Thus if  $\gamma > 0$ and $\gamma\bl{x} \leq \bl{w}^{r}$ for some $\bl{x} \in \#(H)$ then  $\gamma \leq \frac{1}{\v(H) - \alpha(H)}.$ Thus $\Vol_H(\bl{w}^r) \leq  \frac{1}{\v(H) - \alpha(H)},$ and 
$$c_f(H) \geq \sup_{r \in \bb{R}_{+}}\frac{\d(\bl{w}^r)}{\brm{Vol}_{H}(\bl{w}^r)} \geq (\v(H) - \alpha(H)) \sup_{r \in \bb{R}_{+}}\frac{r}{r+1} = \v(H) - \alpha(H),$$
as desired.
\end{proof}

Next we show that $c_f(H)$ is not much smaller than $c(H)$. As in the introduction, given a graph $H$ we define $\mathdefn{\gamma_H} := \sup\{\, \d(G) \: | \:  G \;\text{is complete multipartite}, H \not\preceq G\}$.
	\begin{lem}\label{l:newlower}
For every graph $H$, 
		$$c(H)\geq \gamma_H\geq \sup\{d(\bl{w})|\bl{w}\in \NN_0^{\NN}\setminus\{\bl{0}\},\,\Vol_H(\bl{w})<1\}\geq c_f(H)-2.$$
	\end{lem}
\begin{proof}
	The first inequality follows from the definition of $c(H)$ and the second follows from the fact that $\Vol_{H}(\bl{w})$	is equal to $\Vol_H(G^{\bl{w}})$ whenever $\bl{w}\in \NN_0^{\NN}$ and  $G$ is a complete graph with vertex set $\brm{supp}(\bl{w})$.
	The final inequality is trivial when $c_f(H)\leq 2$, so we may assume $c_f(H)> 2$. Thus it suffices to show that for every $\bl{w}\in \bb{R}_+^{\NN}\setminus\{\bl{0}\}$ with $d(\bl{w})\geq 2$ and $\Vol_H(\bl{w})<1$ there is some $\bl{x}\in \NN_0^{\NN}\setminus\{\bl{0}\}$ such that $\Vol_H(\bl{x})<1$ and $d(\bl{x})\geq d(\bl{w})-2$.

	We assume without loss of generality that for some $m \in \bb{N}$ we have $w_1 \geq w_2 \geq \ldots \geq w_m$ and $w_i=0$ for $i  > m$.
	By LP duality, $\brm{Vol}_{H}(\bl{w}) < 1$ if and only if there exists a vector $\bl{a}  \in \bb{R}_+^{\bb{N}}$ such that $\sum_{i \in \bb{N}}a_iw_i < 1$ and $\sum_{i \in \bb{N}}a_i\mu^{\#}_i  \geq 1$ for every jumbled model  $\mu$ of $H$. 
	
We assume, by reordering components of $\bl{a}$ if needed, that $a_1 \leq a_2 \leq \dots \leq a_m$, as such a reordering does not increase  $\sum_{i =1}^ma_iw_i$.
	
Define $\bl{x} \in \NN_0^{\NN}$, by setting for every $i \in \bb{N}$, 
	$$x_i = \left\lfloor\sum_{j \geq i}w_j\right\rfloor -  \left\lfloor\sum_{j \geq i+1}w_j\right\rfloor.$$
	Note that $|\bl{x}|=\lfloor|\bl{w}|\rfloor\geq 2d(\bl{w})\geq 4$, so in particular $\bl{x}\neq \bl{0}$.
	We claim that $\brm{Vol}_{H}(\bl{x}) < 1$. By the LP duality observation above, it suffices to show that  $\sum_{i \in \bb{N}}a_ix_i < 1$. We have	
	\begin{align*}
		\sum_{i \in \bb{N}}a_ix_i = \sum_{i \in \bb{N}}a_i\left(\left\lfloor\sum_{j \geq i}w_j\right\rfloor -  \left\lfloor\sum_{j \geq i+1}w_j\right\rfloor\right) 
		&= \sum_{i =1}^{m}(a_i-a_{i+1})\left\lfloor \sum_{j \geq i}w_j \right\rfloor \\ 
		&\leq \sum_{i =1}^{m}(a_i-a_{i+1})\left(\sum_{j \geq i}w_j\right) =  \sum_{i =1}^{m}a_iw_i < 1, 
	\end{align*} as claimed.
	
	Next we show that $d(\bl{x}) \geq d(\bl{w})-2$. First, 
	\begin{align*}
	\frac{1}{2}\ip{\bl{x}}{\bl{x}} 
	= \sum_{i \in \bb{N}}x_i \left(\sum_{j \geq i+1}x_j\right)  = \sum_{i =1}^{m-1}x_i \left\lfloor\sum_{j \geq i+1}w_j\right\rfloor 
	& \geq \sum_{i =1}^{m-1}x_i\left(\sum_{j \geq i+1}w_j - 1 \right) \\
	& \geq  \sum_{j =2}^{m}w_j\sum_{i =1}^{j-1}x_i - |x|\\ 
	& = \sum_{j =2}^{m}w_j\left(\left\lfloor\sum_{i \geq 1}w_i\right\rfloor -  \left\lfloor\sum_{i \geq j}w_i\right\rfloor\right)  - |x| \\ 
	&\geq \sum_{j =2}^{m}w_j\left\lfloor\sum_{i =1}^{j-1}w_i\right\rfloor  - |x| \\
	& \geq \sum_{j =2}^{m}w_j\left(\sum_{i = 1}^{j-1}w_i - 1\right)  - |x| \\ 
	&= \frac{1}{2}\ip{\bl{w}}{\bl{w}}  - |x|- |w|.
	\end{align*}        
	It follows that  
	$$d(\bl{x})=\frac{\ip{\bl{x}}{\bl{x}}}{2|x|} \geq \frac{\frac{1}{2}\ip{\bl{w}}{\bl{w}}  - |x|- |w|}{|x|} \geq \frac{\frac{1}{2}\ip{\bl{w}}{\bl{w}}  -  |w|}{|w|}-1 = d(\bl{w})-2,$$
	as desired.
\end{proof}

Say that a vector $\bl{w} \in \{0,1\}^{\bb{N}}$ is an \defn{edge vector} if $|\brm{supp}(\bl{w})|=2$; that is, exactly two components of $\bl{w}$ equal 1, and the rest are 0. Say that a vector $\bl{w} \in \bb{R}^{\bb{N}}_+$ is \defn{matchable} if $w_i \leq \frac{1}{2}|\bl{w}|$ for every $i\in\NN$. The following lemma is straightforward.

\begin{lem}\label{lem:prelim1}
A vector $\bl{w} \in \bb{R}^{\bb{N}}_+$ is matchable if and only if $\bl{w}$ is a linear combination of edge vectors with non-negative coefficients.
\end{lem}

\begin{proof}
	The ``if'' direction is trivial. 	For the ``only if'' direction, without loss of generality, $w_1 \geq w_2 \geq \ldots \geq w_i \geq \ldots$. Since $w_1 -w_2 \leq \sum_{i \geq 3}w_i$, there exists $\bl{z}\in \bb{R}^{\bb{N}}_+$ such that  $\bl{z} \leq \bl{w}, z_1=z_2=0$, and $|\bl{z}|=w_1 -w_2$. 
	Then $\bl{z} + |z|\bl{1}_{1}$  is a linear combination of edge vectors with non-negative coefficients. Let $\bl{w}' =  \bl{w} - \bl{z} - |z|\bl{1}_{1}$. Then $\bl{w}'$   is a matchable vector such that
	$w'_1=w'_2 \geq w'_i$ for every $i \in \bb{N}$. It suffices to show that $\bl{w}'$ is a linear combination of edge vectors with non-negative coefficients. Choose a maximal $\bl{z}' \in \bb{R}^{\bb{N}}_+$ such that $\bl{z}'\leq \bl{w}'$, $z'_1 =z'_2 = 0$ and $\bl{z}'$  is a linear combination of edge vectors with non-negative coefficients. By replacing $\bl{w}'$ with $\bl{w}' - \bl{z}'$ we may further assume that $|\brm{supp}(\bl{w}')| \subseteq \{1,2,3\}$. In this final case, 
	$$\bl{w}' = \frac{w_3}{2}\s{ \bl{1}_{\{1,3\}} + \bl{1}_{\{2,3\}} } + \s{w_1-\frac{w_3}{2}}\bl{1}_{\{1,2\}},$$
 as desired.
\end{proof}

\begin{lem}\label{lem:prelim2}
	Let $H$ be a graph and let $X \subseteq V(H)$. Let $\bl{z},\bl{w} \in  \bb{R}^{\bb{N}}_+$  be such that $\Vol_{H \setminus X}(\bl{w})\geq 1$ and $\bl{z}$ is a matchable vector with $\|\bl{z}\|=2|X|$. Then  $$\Vol_{H}(\bl{w} + \bl{z}) \geq 1. $$
\end{lem}
\begin{proof} For every  $\{i,j\} \subseteq \bb{N}$,  every jumbled model $\mu$ of $H \setminus X$ in $\bb{N}$ can be extended to a jumbled model $\mu^{+}$ of $H$ by setting $\mu^{+}(v) := \{i,j\}$ for each $v\in X$. Thus $\bl{\mu}^{\#} + |X|\bl{1}_{\{i,j\}} \in \# H$. It follows from \cref{lem:prelim1} that $\bl{\mu}^{\#} + \bl{z} \in \# H$ for every matchable $\bl{z} \in  \bb{R}^{\bb{N}}_+$ such that $\|\bl{z}\|=2|X|$, and so
	$\# (H \setminus X) + \bl{z} \subseteq \# H$ for every such $\bl{z}$. Since $\Vol_{H \setminus X}(\bl{w})\geq 1$ there exists $\bl{w}' \in \# H$ such that $\bl{w}' \leq \bl{w}$. Since $\bl{w} + \bl{z} \geq \bl{w}' + \bl{z} \in \# H $, it follows that $\Vol_{H}(\bl{w} + \bl{z}) \geq 1,$ as	desired.
\end{proof}	

\begin{lem}\label{lem:prelim3}
	Let $k \in \bb{N}$, let $H$ be a graph, and let $\bl{z}$ be a matchable vector with $\|\bl{z}\|=2(\v(H)-\alpha_k(H))$. Then  $$\Vol_{H}\s{\frac{\alpha_k(H)}{k}\bl{1}_{[k]} + \bl{z}}\geq 1. $$
\end{lem}
\begin{proof} 
	Let $X \subseteq V(H)$ with $|X| = \v(H)-\alpha_k(H)$ be such that $H' = H \setminus X$ is $k$-colourable. Then there exists a jumbled model $\mu$ of $H'$ in $\bb{N}$ such that
	$\brm{supp}(\mu)=[k]$ and $|\bl{\mu}^{\#}|=\v(H)' = \alpha_k(H)$. By considering a convex combination of vectors obtained from $\bl{\mu}^{\#}$ by shifts of components modulo $k$ we conclude that $\Vol_{H'}\s{\frac{\alpha_k(H)}{k}\bl{1}_{[k]}} \geq 1$.
	The lemma now follows from \cref{lem:prelim2}.
\end{proof}		

\begin{lem}\label{lem:boundedDiff}
For every graph $H$ and vertex $v\in V(H)$,
\begin{equation}\label{e:bd} 
\cx(H) \leq \cx(H \setminus v) + 1.
\end{equation}
\end{lem}

\begin{proof} 
Let $c:= \cx(H \setminus v)$. We show that $\Vol_{H}(\bl{w}) \geq 1$ for every  $\bl{w} \in \bb{R}^{\bb{N}}_+$ with $\|\bl{w} \| >   2(c+1)|\bl{w}|$. This implies the lemma.
	
Define vectors $\bl{x}$, $\bl{y}$ and $\bl{z}$  as follows.
If $\bl{w}$ is matchable, then $\bl{x}=\bl{y}=\s{\frac{1}{2}-\frac{1}{|\bl{w}|}}\bl{w}$.
Otherwise, assume without loss of generality that $w_1$ is the maximum component of $w$. Let $\bl{x}=(w_1-1)\bl{1}_{1}$ and $\bl{y}=\s{1 - \frac{1}{|\bl{w}-\bl{x}|} }(\bl{w}-\bl{x})$.
In both cases, let $\bl{z}=\bl{x}+\bl{y}$.
	
	Routine verification shows that $$\bl{w} =  \s{1+\frac{1}{|\bl{x}|}}\bl{x} +  \s{1+\frac{1}{|\bl{y}|}}\bl{y},$$
	$\bl{w} -\bl{z}$ is matchable and $|\bl{w} -\bl{z}|=2$.

	It thus follows from \cref{lem:prelim2} that if $\Vol_{H \setminus X}(\bl{z})\geq 1$, then $\Vol_{H}(\bl{w})\geq 1$ as desired. Thus we assume  $\Vol_{H \setminus X}(\bl{z})< 1$. Therefore  $\|\bl{z} \| \leq  2c|\bl{z}|$ by the choice of $c$.
	
	By the choice of $\bl{w}$,
	\begin{align*}(2c+2)&(|\bl{z}|+2)  =  2(c+1)|\bl{w}|\\ & < \|\bl{w} \| = \|\bl{z} \| + 2 \ip{\bl{x}+\bl{y}}{\frac{\bl{x}}{|\bl{x}|}+ \frac{\bl{y}}{|\bl{y}|}} +  \left\|\frac{\bl{x}}{|\bl{x}|}+ \frac{\bl{y}}{|\bl{y}|} \right\| \\&\leq 2c|\bl{z}| + 2 \ip{\bl{x}+\bl{y}}{\frac{\bl{x}}{|\bl{x}|}+ \frac{\bl{y}}{|\bl{y}|}} + 4.\end{align*}
	After canceling terms and replacing  $c$ by an upper bound $\frac{\|\bl{x}+\bl{y}\|}{2(|\bl{x}| +|\bl{y}|)}$ in the above, we obtain
	\begin{equation}\label{e:bd1}|\bl{x}|+|\bl{y}|+ \frac{\|\bl{x}+\bl{y}\|}{|\bl{x}| +|\bl{y}|} < \ip{\bl{x}+\bl{y}}{\frac{\bl{x}}{|\bl{x}|}+ \frac{\bl{y}}{|\bl{y}|}}.\end{equation}
	Let $x=|\bl{x}|$ and $y=|\bl{y}|$ for brevity. Rewrite \eqref{e:bd1} as
	\begin{equation}\label{e:bd2}x+y < \ip{\bl{x}+\bl{y}}{\s{\frac{1}{x}-\frac{1}{x+y}}\bl{x}+ \s{\frac{1}{y}-\frac{1}{x+y}}\bl{y}}.\end{equation}
	However,
	\begin{align*}\label{e:bd1}&\ip{\bl{x}+\bl{y}}{\s{\frac{1}{x}-\frac{1}{x+y}}\bl{x}+ \s{\frac{1}{y}-\frac{1}{x+y}}\bl{y}} \\ &\leq \left| \bl{x}+\bl{y}\right|\left| \s{\frac{1}{x} -\frac{1}{x+y}}\bl{x}  +\s{\frac{1}{y}-\frac{1}{x+y}}\bl{y} \right|  \\&= x+y,
	\end{align*}
	contradicting \eqref{e:bd2}. 
\end{proof}

\section{Proof of ~\cref{t:4colored}}\label{s:fourcolored}

\begin{lem}\label{l:normbound} Let $a,b \in \bb{R}_{+}$ and $n \in \bb{N}$  be such that \begin{equation}\label{e:norm1} 
	a^2 \geq kb^2 \qquad \mathrm{and} \qquad 4a \geq 4+  kb^2,
\end{equation}
and let $\bl{w}\in \bb{R}_{+}^{\bb{N}}$ be such that 
\begin{equation}\label{e:norm2} 
a|\bl{w}| - b\sum_{i=1}^k w_i \leq 1.
\end{equation}
Then $\|\bl{w}\|\leq |\bl{w}|$.
\end{lem}                                                

\begin{proof}
	If $b=0$ then $a \geq 1$ by \eqref{e:norm1} and thus $|\bl{w}| \leq 1$, implying $\|\bl{w}\|^2 \leq (|\bl{w}|)^2 \leq |w|$, as desired. Thus we assume $b >0$.
	Let $x = |\bl{w}|$. Then 
	\begin{align*}
	\|\bl{w}\| \leq x^2 - \sum_{i=1}^k w^2_i \leq x^2 - \frac{1}{k}\s{\sum_{i=1}^k w_i}^2 \stackrel{\eqref{e:norm2}}{\leq} x^2 - \frac{1}{k}\fs{1-ax}{b}^2. 
 	\end{align*}
 	Thus it suffices to show that 
 	$$x^2 - \frac{1}{k}\fs{1-ax}{b}^2 \leq x,$$
 	or equivalently
 	\begin{equation}\label{e:norm3} (a^2 - kb^2)x^2 - (2a-kb^2)x +1 \geq 0\end{equation}
 	The discriminant of the quadratic polynomial on the left side of  \eqref{e:norm3} is 
 	$$(2a-kb^2)^2 - 4(a^2 - kb^2) = kb^2(kb^2+4 - 4a)\stackrel{\eqref{e:norm1}}{\leq} 0,$$
 	and its leading coefficient is  non-negative, also by \eqref{e:norm1}. Thus \eqref{e:norm3}
 	holds for all $x\in\RR$, as desired.
\end{proof}

\begin{lem}\label{l:vectors} 
Let 
$$\bl{v}^1=\s{1,\frac{1}{2},0,0},
 \quad \bl{v}^2=\s{\frac{1}{4},\frac{1}{4},\frac{1}{4},\frac{1}{4}} $$ 
 and 
$$
\bl{u}^1=\s{0,2,2,2},  
\quad \bl{u}^2=\s{\frac{2}{3},\frac{2}{3},\frac{2 }{3},\frac{4}{3}},
\quad \bl{u}^3=\s{\frac{2}{5},\frac{6}{5},\frac{6}{5},\frac{6}{5}},
\quad \bl{u}^4=\s{1,1,1,1}. 
$$
Let $\bl{y} \in \bb{R}_{+}^4$ be such that $0 \leq y_1 \leq y_2 \leq y_3 \leq y_4$, $\bl{y} \cdot \bl{v}^1 \geq 1$ and $\bl{y} \cdot \bl{v}^2 \geq 1$. Then there exists $\bl{z} \leq \bl{y}$ such that $\bl{z}$ is a convex combination of $\{\bl{u}^{i}\}.$
\end{lem}

\begin{proof}
	Let $Y \subseteq \bb{R}_{+}^4$ be a polyhedron consisting of all $\bl{y} \in \bb{R}_{+}^4$ satisfying the conditions of the lemma. It suffices to verify that the lemma holds for the vertices of $Y$, which satisfy at least four out of six linear inequalities defining $Y$ with equality. 
	
	The rest of the proof is a routine case analysis of possible quadruples of tight inequalities. If $y_1=0$ then $y_2 \geq 2$ and so $\bl{y} \geq \bl{u}^1$, as desired. Thus we assume $y_1 > 0$. If $y_1=y_2$ and $ \bl{y} \cdot \bl{v}^1 = 1$ then $y_1=y_2 = \frac23 $, and $y_3+y_4\geq \frac83$, implying that $\bl{y} \geq \bl{u}^2$. Otherwise, $y_2 = y_3 = y_4$, $\bl{y} \cdot \bl{v}^2 = 1$, and either $y_1=y_2$, in which case $\bl{y} = \bl{u}^4$, or $ \bl{y} \cdot \bl{v}^1 = 1$, in which case $\bl{y}=\bl{u}^3$.
\end{proof}

We are now ready to prove \cref{t:4colored} which we restate for convenience.

\Fourcolored*

\begin{proof} 
	Let $c := \max \s{\frac{\v(H)}{2},\tau(H)}$.	It suffices to show that
\begin{equation}\label{e:bip1} 
2c|\bl{w}|  \geq  \|\bl{w}\|	
\end{equation}
for every $\bl{w} \in \bb{R}^{\bb{N}}_+$ such that $\Vol_H(\bl{w}) \leq 1$. Assume without loss of generality that $\brm{supp}(\bl{w})=[n]$ for some $n \in \bb{N}$. 
 
Since $\{\bl{w}' \in \bb{R}^{n}_+ \: | \: \Vol_H(\bl{w}') \geq  1 \}$	is convex, there exists $b \in \bb{R}$  and $\bl{x} \in \bb{R}^{n} \setminus \{\bl{0}\}$ such that  $\bl{x} \cdot \bl{w} \leq b$ and $\bl{x} \cdot \bl{w}' \geq b$ for every $\bl{w}' \in \bb{R}_+^{n}$ such that $ \Vol_H(\bl{w}') \geq 1$. The last condition implies that  $\bl{x} \in \bb{R}_+^{\bb{N}}$. Thus $b > 0$ and by rescaling $\bl{x}$ we assume (with foresight) that  $b=2c$ and that $0 \leq x_1 \leq x_2 \leq \dots \leq x_n$.

If $n=1$ then $\|\bl{w}\|=0$ and \eqref{e:bip1} trivially holds. Now assume $n \geq 2$. 
Let $\bl{v}^1 :=\s{1,\frac{1}{2},0,0,\dots}$. Then $2c \bl{v}^1 \geq (\v(H),\tau(H))$, and by  \cref{lem:prelim3} with $k=1$,  
$$ \Vol_H(2c\bl{v}^1) \geq \Vol_H(\alpha(H)\bl{1}_{\{1\}} + (1 - \alpha(H))\bl{1}_{\{1,2\}}) \geq 1.$$
Thus $\bl{x} \cdot (2c \, \bl{v}^1) \geq 2c$, and so $\bl{x} \cdot   \bl{v}^1 \geq 1$. 

Suppose first that $n \geq 4$. 
Let $\bl{v}^2 :=\s{\frac{1}{4},\frac{1}{4},\frac{1}{4},\frac{1}{4}}$.
Similar to the above, since $\alpha_4(H)=\v(H) \leq 2c$,   \cref{lem:prelim3} applied with $k=4$ implies $\Vol_H(2c\,\bl{v}^2) \geq 1$, and thus $\bl{x} \cdot   \bl{v}^2 \geq 1$.
Define $\bl{y} \in \bb{R}_+^{n}$ by $y_i := x_i$ for $i\in[3]$ and $y_i := x_4$ for $i\in[4, n]$. Then $\bl{y} \leq \bl{x}$  and $\bl{y} \cdot   \bl{v}^i =\bl{x} \cdot   \bl{v}^i \geq 1$ for $i\in[2]$. Let $\bl{u}^1, \dots, \bl{u}^4 \in \bb{R}_+^{n}$ be given by 
\begin{align*}
\bl{u}^1 & :=\s{0,2,2,\dots,2},\quad   & 
\bl{u}^2 & :=\s{\frac{2}{3},\frac{2}{3},\frac{2 }{3},\frac{4}{3},\dots,\frac{4}{3}}, \\
\bl{u}^3 & :=\s{\frac{2}{5},\frac{6}{5}, \dots, \frac{6}{5}}, \quad & 
\bl{u}^4 & :=\s{1,1,\dots,1}. 
\end{align*}
By \cref{l:vectors} applied to the projections of the above vectors to the first four cooordinates, there exists $\bl{z}\in \bb{R}_+^{n}$ such that $\bl{z} \leq \bl{y}$ and $\bl{z}$ is a convex combination of $\{\bl{u}^{i}\}$. It follows that $\bl{u}^i \cdot \bl{w} \leq 2c$ for some $i \in [4]$. Let $\bl{w}' := \frac{1}{2c} \bl{w}$. Then $$\bl{u}^i \cdot \bl{w}' \leq 1.$$

If $i=1$ then we can rewrite this inequality as $2|\bl{w}'|-2w'_1 \leq 1$. Thus $\bl{w}'$ satisfies \eqref{e:norm2} with $a=b=2$ and $k=1$. Since \eqref{e:norm1} holds for this $a,b$ and $k$, \cref{l:normbound} implies that $|\bl{w}'| \geq \|\bl{w}'\|$, implying \eqref{e:bip1} as desired.

For the remaining values of $i$, the same argument applies, but with different choices of $a,b$ and $k$. Namely, we choose $(a,b,k)= (\frac43 ,\frac23 ,3)$ for $i=2$; $(a,b,k)= (\frac65,\frac25,2)$  for $i=3$, and $(a,b,k)= (1,0,1)$ for $i=4$. 

It remains to consider the case $n \in \{2,3\}$. We deal with these cases simultaneously; that is, we assume $n=3$, but relax the assumption $x_3 >0$. Since $\bl{x} \cdot   \bl{v}^1 \geq 1$, we have $x_1 + \frac{x_2}{2} \geq 1$. Thus $x_3 \geq x_2 \geq 1 - 2x_1$, implying that $\bl{x} \geq (x,2-2x,2-2x)$ for some $x \leq \frac23 $. Defining $\bl{w}'$ as in the previous case, $x|\bl{w}'| + (2-3x)(|\bl{w}'| - w'_1) \leq 1$. Thus either $|\bl{w}'| - w'_1 \leq \frac12$, in which case  $|\bl{w}'| \geq \|\bl{w}'\|$ by  \cref{l:normbound}, as before, or 
$|\bl{w}'| \leq \frac{3}{2}$. In this last case, 
$$\|\bl{w}'\| = |\bl{w}'|^2 - (w'_1)^2 - (w'_2)^2 - (w'_3)^2 \leq \frac{2}{3}|\bl{w}'|^2 \leq |\bl{w}'|,$$
as desired.
\end{proof}

We are now ready to prove \cref{t:AlphaThreeUpperBound} which we restate for convenience.

\AlphaThreeUpperBound*

\begin{proof}
	Let $X$ be an independent set of $H$ of size $\alpha(H)$, and let $Y$ be a set of $\alpha_3(H)$ vertices inducing a $3$-colourable subgraph in $H$.
	Let $H' := H[X \cup Y]$. Then $\tau(H') = \v(H') - \alpha(H')=\v(H')-\alpha(H)$. 
	Note that $\chi(H') \leq 4$. By \cref{t:4colored}, 
	$$c_f(H') \leq \max\left( \frac{\v(H')}{2}, \v(H')-\alpha(H) \right).$$
	By \cref{lem:boundedDiff}, 
	\begin{align*}
	c_f(H) \leq c_f(H') + (\v(H)-\v(H')) \leq & \max\left(\v(H) - \frac{\v(H')}{2}, \v(H) -\alpha(H)\right)\\
	\leq & \max\left(\v(H) - \frac{\alpha_3(H)}{2}, \tau(H)\right).
	\end{align*}
	The result follows from \cref{t:main1}.
\end{proof}

%%%%%%%%%%%%%%%%%%%%%%%%%%%%%%%%%%%%%%%%%%%%%%%%%%%%%%%%%%%%%%%%%
\section{Proof of \cref{t:twothirds}}\label{s:twothirds}

\begin{lem}\label{l:technical0}
	Let $H$ be a graph, and for all $i\in \mathbb{N}$ let $\alpha_i := \frac{\alpha_i(H)}{\v(H)}$ for brevity, and let $c \geq 1 - \alpha_1$ be real such that for all $j\in[2,i]$,
\begin{align}\label{e:ld0} 
c &\geq  1 - \alpha_1, \notag\\
	c&(4 -\alpha_1 - \alpha_i) \geq \alpha_i(2-\alpha_1) +
	(1-\alpha_i)(3-2\alpha_1), 
	\\
	c&(4 -\alpha_i - \alpha_j) \geq \s{2 -\alpha_j}\s{\frac{i-1}{i}\alpha_i + \frac{j-1}{j}(2-2\alpha_i)}.\notag
	\end{align}
Then $c_f(H) \leq c \v(H)$.
\end{lem}

\begin{proof}
	It suffices to show that for every $\bl{w} \in  \bb{R}^{\bb{N}}_+$, 
	\begin{equation}\label{e:ld1} 
	2c \: \v(H)|\bl{w}| \Vol_H(\bl{w}) \geq  \|\bl{w}\|	.
	\end{equation} 
Assume without loss of generality that $w_1\geq w_2 \geq \dots \geq w_k > 0$ and that $w_i=0$ for $i > k$.
	
	First, consider the case $(2 - \alpha_1) w_1 \geq |\bl{w}|$. Let $w_0 :=  |\bl{w}| - w_1$. Then $\Vol_H(\bl{w}) \geq \frac{w_0}{\v(H)-\alpha(H)}$ and $\|\bl{w}\| \leq 2|\bl{w}|w_0$. 
	Substituting these bounds reduces \eqref{e:ld1} to $c \geq 1 - \alpha_1$, which holds by the choice of $c$.
	
Now assume $(2 - \alpha_1) w_1 < |\bl{w}|$. For  $\bl{z} \in  \bb{R}^{\bb{N}}_+$, define 
	$$\xi(z) :=\sum_{i=1}^{k}\frac{1-\alpha_i}{\alpha_i}i(z_i - z_{i+1}).$$
Choose $\bl{z} \in  \bb{R}^{\bb{N}}_+$ so that:
	\begin{itemize}
		\item[(i)]  there exists $\ell\in[2,k]$ such that $z_1=z_2=\dots=z_\ell \leq w_l$, and $z_{i}=w_i$ for every $i > \ell$,
		\item[(ii)]  $|\bl{w} - \bl{z}| \geq 2\xi(z)$, and
		\item[(iii)] $|\bl{w} - \bl{z}| \geq (2 -\alpha_1)(w_1 -z_1) + \alpha_1\xi(z)$.
	\end{itemize}  
	and subject to these conditions, $|\bl{z}|$ is maximum. Such a choice is possible since $\bl{z}=\bl{0}$ satisfies (i), (ii) and (iii).	By the maximality of $z$, at least one (ii) and (iii) holds with equality.
		
	Now define $\bl{x},\bl{y} \in  \bb{R}^{\bb{N}}_+$ such that $\bl{w}=\bl{x}+\bl{y}+\bl{z}$, and
		\begin{itemize}
			\item[(iv)]  $\bl{y}$ is a matchable vector with 
			$
			|\bl{y}| = 2\xi(\bl{z}),  
			$
			\item[(v)] $\bl{x} = \bl{0}$, or $(2 - \alpha_1) x_1 = |\bl{x}|$
			and $y_1=\frac{1}{2}|\bl{y}|$, 
		\end{itemize}  
	as follows.	
	Let $\bl{w}' := \bl{w} - \bl{z}$. If $\bl{w}'$ is matchable, then  $2(w_1 - z_1) \leq  |\bl{w} - \bl{z}|$, and so if the inequality in (ii) is strict, then so is the one in (iii). Thus (ii) holds with equality; that is, $|\bl{w} - \bl{z}| = 2\xi(\bl{z})$, and $\bl{x}=\bl{0}$ and $\bl{y}=\bl{w} - \bl{z}$ satisfy (iv) and (v).
	
	Thus we assume that $2w'_1 > |\bl{w}'|$. Then similarly to the previous case we deduce that  $|\bl{w}'| = (2 -\alpha_1)w'_1 + \alpha_1\xi(\bl{z})$. Let 
	$$\bl{y} := \xi(\bl{z})\bl{1}_{[1]} + \frac{\xi(\bl{z})}{|\bl{w}'| - w'_1}(\bl{w}' - w'_1\bl{1}_{[1]}),$$ 
	and let $\bl{x} := \bl{w}' - \bl{y}$. Then $|\bl{y}| = 2\xi(\bl{z})$, and $y_1 = \frac{1}{2}|\bl{y}|$. Thus $\bl{y}$ is  matchable and (iv) holds.
Since $2w_1'  > |\bl{w}'| \geq 2\xi(\bl{z})$ we have $|\bl{w}'|-w'_1 =  (1 -\alpha_1)w'_1 + \alpha_1\xi(\bl{z}) \geq \xi(\bl{z})$. Thus $\bl{y} \leq \bl{w}'$, and so $\bl{x} \geq \bl{0}$. Moreover, $$(2 -\alpha_1)x_1 = |\bl{w}'| - 2 \xi(\bl{z}) = |\bl{w}'| - |\bl{y}| =|\bl{x}|. $$
Thus (v) also holds. This finishes our construction of $\bl{x}$ and $\bl{y}$. 
	
	It follows from (v) that $$\ip{\bl{x}}{\bl{y}} \leq |\bl{x}||\bl{y}|-x_1y_1 = \s{1 - \frac{1}{2(2-\alpha_1)}}|\bl{x}||\bl{y}|$$
	For $i\in[\ell,k]$, let $\bl{z}^i := \frac{\alpha_i \v(H)}{i}\bl{1}_{[i]}$ and let $\bl{y}^i := 2(1-\alpha_i)\,\v(H)\frac{\bl{y}}{|\bl{y}|}$. By \cref{lem:prelim3}, $\Vol_H(\bl{z}^i + \bl{y}^i) \geq 1$. Let $s_i := \frac{i (z_i - z_{i+1})}{\alpha_i \v(H)}$. Since $\bl{y}+\bl{z}=\sum_{i=\ell}^{k} s_i(\bl{z}^i + \bl{y}^i)$, we have $\Vol_H(\bl{y}+\bl{z}) \geq \sum_{i=\ell}^{k} s_i$. Let $s_1 := \frac{x_1}{\v(H)}$. Then $\Vol_H(\bl{x}) \geq s_1$ and $|\bl{x}| = s_1(2-\alpha_1)\,\v(H)$.

	 Combining\begin{align*}
	\Vol_H(\bl{w}) &\geq s_1 + \sum_{i=\ell}^{k} s_i,\\
	|\bl{w}|&= \v(H) \cdot \s{(2 - \alpha_1)s_1 + \sum_{i=\ell}^{k}(2-\alpha_i)s_i}
	\end{align*}
	Next we upper bound $\|\bl{w}\|$.
	For $i,j\in[\ell,k]$ with  $j \leq i$, 
	\begin{align*}
	\ip{\bl{z}^i}{\bl{z}^j+\bl{y}^j} &=
 \frac{i-1}{i}|\bl{z}^i||\bl{z}^j+\bl{y}^j|=\frac{i-1}{i}\alpha_i(2-\alpha_j) \v^2(H),  \\
	\ip{\bl{y}^i}{\bl{z}^j+\bl{y}^j} &\leq \frac{j-1}{j}|\bl{y}^i||\bl{z}^j|+\frac{\ell-1}{\ell}|\bl{y}^i||\bl{y}^j| \leq \frac{j-1}{j} (2-2\alpha_i)(2-\alpha_j)\v^2 (H),  \\
	\ip{\bl{x}}{\bl{z}^i+\bl{y}^i} &\leq |\bl{z}^i||\bl{x}|+\s{1-\frac{1}{2(2-\alpha_1)}} |\bl{y}^i||\bl{x}| \\ &= \s{\alpha_i(2-\alpha_1) +
	(1-\alpha_i)(3-2\alpha_1)}\v^2(H)s_1, \\
\ip{\bl{x}}{\bl{x}} &\leq (3-\alpha_1)(1-\alpha_1)\v^2(H)s^2_1 .
	\end{align*}
	Therefore,
	\begin{align*}
		\|\bl{w}\| 
		&= \ip{\bl{x}}{\bl{x} + 2\sum_{i=\ell}^{k}s_i(\bl{z}^i+\bl{y}^i)} \\
		& \quad + \sum_{j=\ell}^{k}s_j\ip{(\bl{z}^j+\bl{y}^j)}{s_j(\bl{z}^j+\bl{y}^j)+ 2\sum_{i=j+1}^{k}s_i(\bl{z}^i+\bl{y}^i)} \\ 
		&\leq \v^2(H) \s{ s_1 \s{s_1(3-\alpha_1)(1-\alpha_1) + 2\sum_{i=l}^{k}s_i\s{\alpha_i(2-\alpha_1) +
					(1-\alpha_i)(3-2\alpha_1)} }}\\	
	& \quad + \sum_{j=l}^{k}s_j(2-\alpha_j)\s{s_j\frac{j-1}{j}(2-\alpha_j) +  2\sum_{i=j+1}^{k}s_i\s{\frac{i-1}{i}\alpha_i + \frac{j-1}{j} (2-2\alpha_i)}}.
		\end{align*}
	Substituting the above bounds on $|\bl{w}|,\Vol_H(\bl{w})$ and $\|\bl{w}\|$ in to \eqref{e:ld1}, we see that by \eqref{e:ld0} in the resulting inequality the coefficient of $s_is_j$ on the left side is at least as large as the  coefficient on the left side for all $i,j \in \{1\} \cup [\ell,k]$. Hence \eqref{e:ld1} holds, as desired.
	\end{proof}

\begin{lem}\label{l:technical2}
	Let $2 \leq j \leq i$ be integers, and let $ 0 \leq \alpha_j \leq \alpha_i \leq 1$, and $c$ be real such that
	\begin{equation}\label{e:c}
	c \geq \max \left\{ \frac{2}{3}, \frac{i-1}{i}\s{1 -\frac{\alpha_i}{2}}, \frac{j-1}{j}\s{1 -\frac{\alpha_j}{2}}\right\}.
	\end{equation}
	Then \begin{equation}\label{e:localdensity}
	c(4 -\alpha_i - \alpha_j) \geq \s{2 -\alpha_j}\s{\frac{i-1}{i}\alpha_i + \frac{j-1}{j}(2-2\alpha_i)}.
	\end{equation}
\end{lem}

\begin{proof}
	Let $x := 2 -\alpha_j$ and $y := 2 - \alpha_i$. Then $1 \leq y \leq x \leq 2$, and \eqref{e:c} can be rewritten as 
	\begin{equation}\label{e:c2}
	c \geq \max \left\{ \frac{2}{3}, \; \frac{i-1}{2i}y, \: \frac{j-1}{j}x \right\}.
	\end{equation}
	
	Suppose for a contradiction that \eqref{e:localdensity} does not hold. Expressing $\alpha_i$ and $\alpha_j$ in terms of $x$ and $y$, we thus have
		\begin{equation}\label{e:ld2}
		c(x+y) < x \s{y - \frac{2-y}{i} - \frac{2y-2}{j}}. 
		\end{equation}	
	By \eqref{e:c2} $$ c(x+y) \geq \s{\frac{i-1}{2i}y} x  +  \s{\frac{j-1}{2j}x}y = \s{1 - \frac{1}{2i} - \frac{1}{2j}}xy,$$	and so \eqref{e:ld2} implies
	$$ \s{1 - \frac{1}{2i} - \frac{1}{2j}} y < \s{1 + \frac{1}{i} - \frac{2}{j}}y - \frac{2}{i} + \frac{2}{j},$$
	which in turn implies 
		\begin{equation}\label{e:y}
	y < \frac{4}{3}. 
	\end{equation}
		
	Secondly, as $c \geq  \frac{j-1}{2j}x$,  \eqref{e:ld2} similarly implies 	$$ \frac{j-1}{2j}(x+y) < \s{1 + \frac{1}{i} - \frac{2}{j}}y - \frac{2}{i} + \frac{2}{j},$$
	which simplifies to \begin{equation}\label{e:x0} x < y + \frac{2(i-j)}{i(j-1)}(2-y).\end{equation}
	Since $ i \geq j$ and $y \leq 2$, \eqref{e:x0} in turn implies  \begin{equation}\label{e:x00} x < y + \frac{2}{(j-1)}(2-y) = \frac{j-3}{j-1}y + \frac{4}{j-1}.\end{equation}
	Finally, since $c \geq \frac23 $ and $y \leq 2$, \eqref{e:ld2}  implies
	\begin{equation}\label{e:ld3}
	\frac{2}{3}(x+y) - x \s{\s{1 - \frac{2}{j}}y + \frac{2}{j}} < 0
	\end{equation}	
	
	Suppose first that $j \geq 3$.
Since $y \geq 1$, 
	$$\frac{2}{3} - \s{1 - \frac{2}{j}}y + \frac{2}{j} < 0.$$
	Thus the left side of \eqref{e:ld3} decreases with $x$, and so by \eqref{e:x00}, it suffices to show that \eqref{e:ld3} does not hold when $x = \frac{j-3}{j-1}y + \frac{4}{j-1}$. Substituting this value of $x$ into \eqref{e:ld3} we obtain	\begin{equation*}
	\frac{2}{3}\s{\frac{2j-4}{j-1}y +  \frac{4}{j-1}} - \s{\frac{j-3}{j-1}y + \frac{4}{j-1}} \s{\s{1- \frac{2}{j}}y + \frac{2}{j}} < 0,
	\end{equation*}	
	which luckily simplifies to 
	$$(4-3y)(j-3)(2+(j-2)y) < 0.$$
 Since $j \geq 3$ and $1 \leq y < \frac43 $ by \eqref{e:y}, this last inequality yields the desired contradiction.
	
	It remains to consider the case $j=2$. In this case, \eqref{e:ld2} simplifies to $\frac{2}{3}y -\frac{1}{3}x <0$, which is again a contradiction since $y \geq 1$ and $x \leq 2$.
	\end{proof}	

With the technical aspects of the proof of \cref{t:twothirds} handled by \cref{l:technical0,l:technical2}, we now finish the argument. We restate the theorem for convenience.

\Twothirds*

\begin{proof}
	Let  $\alpha_i := \frac{\alpha_i(H)}{\v(H)}$, and let $c := \max \{ \frac23 , \frac{c_T(H)}{\v(H)} \}$. Then 	for every $i\in\NN$, 
	$$c \geq \max\left\{ \frac{2}{3}, 1 - \alpha,  \frac{i-1}{i}\left(1-\frac{1}{2}\alpha_i \right)\right\}.$$
It suffices to show that $c_f(H) \leq c\v(H)$. By \cref{l:technical0} in turn it suffices to show that the inequalities \eqref{e:ld0} hold.
	The first holds since  $c \geq 1 -\alpha$, and the last holds by \cref{l:technical2}. 
It remains to show that 
	\begin{equation}
	\label{e:23}c(4 -\alpha_1 - \alpha_i) \geq \alpha_i(2-\alpha_1) +
	(1-\alpha_i)(3-2\alpha_1).
	\end{equation}  
	We have $$c(4 -\alpha_1 - \alpha_i) \geq \frac{2}{3}(3-\alpha_1) + (1-\alpha_1)(1-\alpha_i),$$ and so \eqref{e:23} is implied by
	$$\frac{2}{3}(3-\alpha_1) + (1-\alpha_1)(1-\alpha_i) \geq  \alpha_i(2-\alpha_1) +
	(1-\alpha_i)(3-2\alpha_1),$$
	which reduces to $\frac{\alpha_1}{3} \geq 0$ after cancellations.
\end{proof}

We finish this section by deriving from  \cref{t:twothirds}  a bound on the fractional extremal function of $H$ in terms of the maximum of $\frac{\tau(H')}{\v(H')}$ taken over non-null subgraphs $H'$ of $H$. This bound is used in the proof of \cref{t:upper0}.

\begin{cor}\label{c:tau}
	Let $H$ be a graph, and let $r \geq \frac23$ be such that for every  non-null subgraph $H'$ of $H$.
	\begin{equation}\label{Hprime}
	\tau(H') \leq r\,\v(H').
	\end{equation}
	Then $c_f(H) \leq r\,\v(H)$.
\end{cor}

\begin{proof}If $c_f(H) \leq  \frac{2}{3} \v(H)$ then the corollary clearly holds. Thus we assume that $c_f(H) \geq  \frac{2}{3} \v(H)$.
	By \cref{cfLowerBound} we have  $c_f(H) - c_T(H) \geq \tau(H) \geq \frac{2}{3}{\v(H)}$. Thus by \cref{t:twothirds} it suffices to show that for every $k\in\NN$. 
	\begin{equation}\label{e:tau} 
		\frac{k-1}{k}\left(\v(H)-\frac{1}{2}\alpha_k(H) \right) \leq  r\v(H).
	\end{equation}
Note that  $\alpha(H') \geq (1 - r)\,\v(H')$ for every subgraph $H'$ of $H$ by \eqref{Hprime}. It follows by induction on $k$ that $\alpha_k(H) \geq (1-r^k) \v(H)$. After substitution of this bound into \eqref{e:tau}  it suffices to show that
	\begin{equation}\label{e:tau2}  
	\frac{k-1}{k}\left(1-\frac{1}{2}(1-r^k) \right) \leq r.
	\end{equation}
Since the left side of \eqref{e:tau2} is a convex function of $r$ and $r \in [\frac23 ,1]$ it suffices to verify \eqref{e:tau2} for $r =1$ and $r=\frac23 $. The case $r=1$ is trivial, while for $r=\frac23 $, we may assume $k\geq 2$ and  \eqref{e:tau2} reduces to
	$1+\fs{2}{3}^k \leq \frac{4k}{3(k-1)}$. 
Since  $1+ \fs{2}{3}^3 \leq \frac43 $, this last inequality holds for $k \geq 3$, and it clearly holds for $k=2$.
\end{proof}

\section{Connectivity and density gain}\label{s:connect}

Proving \cref{t:main1} requires us to show that for every $\eps > 0$ if graphs $G$ and $H$ satisfy $\d(G) \geq (1+\eps)c_f(H)$ and $\v(H)$ is large enough as a function of $\eps$,  then $H$ is a minor of $G$.
The goal of this section is to reduce the proof of \cref{t:main1} to the ``dense'' case, more specifically, to the case when $G$ has connectivity linear in $\v(G)$. Then  \cref{addedges} shows, roughly, that if we find a graph $H'$  obtained from $H$ by deleting a sublinear number of edges, as a minor in $G$, we can use the connectivity to restore the deleted edges. This result allows us to reduce the proof of \cref{t:main1} to the case when $H$ has bounded 
component size, which reduces to \cref{l:copies}.

Our first lemma is used to ``stitch'' together subgraphs of highly connected graph, restoring a fraction of connectivity.

\begin{lem}\label{l:restoreConnectivity}
For $k,n\in\NN$, let $G_0, G_1,\dots,G_k$ be vertex-disjoint subgraphs of a graph $G$ such that $\v(G_0) \geq kn$, $\kappa(G) \geq kn$ and $\kappa(G_i) \geq n$ for $i\in[0,k]$. Then there exists  a minor $G'$ of $G$ such that $\kappa(G') \geq n$ and $G_0 \cup G_1 \cup \dots \cup G_k$ is isomorphic to a spanning subgraph of $G'$.
\end{lem}

\begin{proof}
	Let $X_i := V(G_i)$ for $i\in[0,k]$. 	By Menger's theorem there exist linkages $\mc{P}_1,\dots, \mc{P}_k$ in $G$ with the following properties: 
	\begin{itemize}
		\item $|\mc{P}_i|=n$ for every  $i\in[k]$, 
		\item $\mc{P} :=\mc{P}_1 \cup \dots \cup \mc{P}_k$ is a linkage in $G$,
		\item For every $i\in[k]$, every $P \in \mc{P}_i$ has one end in $X_0$ and the other end in $X_i$.
	\end{itemize}

	Let $G'' := (\bigcup_{i=0}^k G_i) \cup (\bigcup_{P \in \mc{P}} P)$. We obtain a minor $G'$ of $G''$ by repeatedly contracting edges of the paths in $\mc{P}$ that have at least one end in $V(G'') - \bigcup_{i=0}^k X_i $. We claim that $G'$ is as required. By construction,  $G_0 \cup G_1 \cup \dots \cup G_k$ is isomorphic to a spanning subgraph of $G'$. Moreover, slightly abusing the notation, we may assume that $\mc{P}_1,\dots,\mc{P}_k$ are the linkages in $G'$ satisfying all the conditions listed above. 
	
	It remains to show that $\kappa(G') \geq n$. Suppose not. Then there exists a separation $(A,B)$ of $G'$ of order less than $n$. Since $\kappa(G_0) \geq n$, we may assume without loss of generality that $X_0 \subseteq A$. Further, $X_i \cap (B-A) \neq \emptyset$ for some $i\in[k]$. Since   $\kappa(G_i) \geq n$ it follows that $X_i \subseteq B$. Thus every path $P \in \mc{P}_i$ has one end in $A$ and the other in $B$, implying that $V(P) \cap A \cap B \neq \emptyset$. However, $|\mc{P}_i|=n > |A \cap B|$, a contradiction.
\end{proof}

The next lemma accomplishes the most technical step towards the goal of this section.

\begin{lem}
	\label{FindG'}
	For all $C\geq 5$ and all $s,n \in\NN_0$, such that $s \geq 9Cn$, and for every graph $G$ such that $\v(G)\geq Cs$, $\kappa(G) \geq 3Cn$, $\d(G)\leq 2s$, and every edge of $G$ is in at least $s$ triangles, there exists a minor $G'$ of $G$ such that $Cs \leq \v(G') \leq 2(C+1)s$, $\kappa(G')\geq n$ and $\delta(G')\geq s-3Cn$. 
\end{lem}

\begin{proof} If $\v(G) \leq 2(C+1)s$ then $G'=G$ satisfies the lemma and so we assume that $\v(G) \geq 2(C+1)s$. Let $S := \{v \in V(G) \, | \, \deg(v) \leq  4s -1 \}$. Then $2s\,\v(G) \geq \e(G) \geq 4s(\v(G)-|S|)$, and therefore $|S| \geq \v(G)/2 \geq (C+1)s$. 
	
	For $u \in V(G)$, let $G(u)$ denote the subgraph of $G$ induced by $N[u]$. 
	Let $S' \subseteq S$ be chosen minimal such that $ \v (\cup_{u \in S'} G(u)) \geq (C+1)s$. (Such a choice is possible as $S'=S$ satisfies the  inequality.)
	Since $|S| \geq Cs$, and $\v(G(u)) \leq 4s$ for every $u \in S$, it follows that $$(C+1)s \leq \v (\cup_{u \in S'} G(u)) \leq (C+5)s.$$ 
	Let $G'' :=\cup_{u \in S'} G(u)$. Then $\delta(G'') \geq s$, since $\delta(G(u)) \geq s$ for every $u \in V(G)$. 
	
	Choose a maximum integer  $k\in[0,3C]$ so that  $G''$ contains vertex-disjoint subgraphs $G_0,G_1,\dots,G_k$, where $\delta(G_i) \geq s-kn$ for every $i\in[0,k]$, and $\sum_{i=0}^k \v(G_i) \geq (C+1)s -kn$. (Such a choice is possible, since $G_0=G''$ satisfies the above conditions for $k=0$.)\ Note that $k \leq 3C-1$, as otherwise 
	$$\sum_{i=0}^k \v(G_i) \geq 3C(s-kn) \geq 3C(s-3Cn) > 3C \cdot \frac{2}{3}s = 2Cs \geq (C+5)s$$ since $C \geq 5$ and $9Cn \leq s$.
	Note further that $\sum_{i=0}^k v(G_i) \geq (C+1)s -3Cn \geq Cs$.
	
	Next we verify that $G_0, G_1,\dots,G_k$ satisfy the conditions of \cref{l:restoreConnectivity}. We have $\v(G_0) > \delta(G_0) \geq s-3Cn \geq 6Cn \geq kn$,
	and  $\kappa(G) \geq kn$.
	It remains to verify that $\kappa(G_i) \geq n$ for every $i\in[0,k]$. Suppose not. Then, without loss of generality, $G_k$ admits a separation $(A,B)$ of order at most $n$. In such a case, $G_0,G_1,\dots,G_{k-1},G_k[A\setminus B],G_k[B\setminus A]$ is a collection of $k+2$ subgraphs of $G''$, which contradicts the choice of $k$.
	
	By \cref{l:restoreConnectivity}, $G''$ has a minor $G'$ such that $\kappa(G') \geq k$ and  $G_0 \cup G_1 \cup \dots \cup G_k$ is isomorphic to a spanning subgraph of $G'$. Thus $G'$ is as required.
\end{proof}

To be able to apply \cref{FindG'} we need to replace $G$ by a graph with reasonably high connectivity, such that every edge belongs to many triangles, while losing only a small fraction of density. The argument accomplishing this is due to Mader~\cite{Mader68}, and is part of the standard toolkit in the area. We include the proof for completeness.

\begin{lem}[{\cite{Mader68}}]\label{Mader}
	Let $G$ be a graph and let $d,k\in\NN_0$ such that $\d(G) \geq d \geq 2k$. Then there exists a minor $G'$ of $G$ such that $\kappa(G') \geq k$, $d \geq \d(G') \geq d-k$,  and every edge of $G'$ lies in at least $d$ triangles.
\end{lem}

\begin{proof}
	Let $\mc{E}_{d,k}$ be the class of graphs $G''$ such that $\v(G'') \geq d$ and $\e(G'') \geq d\,\v(G'') - kd$. We show that  $\kappa(G') \geq k$, $\d(G') \geq d-k$,  and every edge of $G'$ lies in at least $d$ triangles for every minor-minimal graph $G'$ in $\mc{E}_{d,k}$. Since $G$ is an element of $\mc{E}_{d,k}$ this will imply the lemma.
	
	First note that $d \geq \d(G') \geq  d-k$, since  $\e(G') = d\,\v(G') - kd$. Second,
	note that $\v(G') > d$, as otherwise $\e(G') < \frac{d}{2} \,\v(G') \leq d\,\v(G') - kd$, contradicting the assumption that $G'\in \mc{E}_{d,k}$. Thus if some edge $e \in E(G')$ lies in at most $d-1$ triangles, then contracting $e$ gives a graph in $\mc{E}_{d,k}$, contradicting the choice of $G'$.
	
	Finally, let $k \geq 1$ and suppose for a contradiction that there exists a separation $(A,B)$ of $G'$ with order at most $k-1$. Then $G'[A]$ and $G'[B]$ are not members of $\mc{E}_{d,k}$ by the choice of $G'$. Therefore,
	$\e(G'[A]) \leq d|A|-kd$ and $\e(G'[B]) \leq d|B|-kd$. Summing these inequalities, 
	$$\e(G) \leq \e(G'[A])+\e(G'[B]) \leq d(|A|+|B|) -2kd < d(\v(G')+k) -2kd \leq d\,\v(G')-kd,$$
	a contradiction.
\end{proof}	  

Combining \cref{FindG',Mader} we obtain the following.

\begin{cor}\label{c:denser}
For every $\eps, C > 0$ there exists $\eps'=\eps'_{\ref{c:denser}}(\eps,C) > 0$ such that  every graph $G$  contains a minor $G'$ such that $\kappa(G') \geq \eps' \v(G')$, and either:
\begin{itemize}
	\item[(i)] $\v(G') \geq C\d(G)$ and $\delta(G') \geq (1 -\eps)\d(G)$, or
	\item[(ii)] $\d(G') \geq (1 - \eps)\d(G)$.
\end{itemize}
\end{cor}		

\begin{proof} 
	Without loss of generality, assume that $\eps \leq \frac12$ and $C \geq 5$. We show that $\eps' = \frac{\eps}{6C(C+1)}$ satisfies the lemma.
	
	We further assume that $\d(H) \leq \d(G)$ for every minor $H$ of $G$, as otherwise we can replace $G$ by $H$. Let $d  := \d(G)$. 
	By \cref{Mader} applied with $k = \eps d$  there exists a minor $G_1$ of $G$ such that $\kappa(G_1) \geq \eps d$, $\d(G_1) \geq (1 - \eps)d$,  and every edge of $G_1$ lies in at least $d$ triangles. 
	If $\v(G_1) \leq Cd$, then $\kappa(G_1) \geq \frac{\eps}{C} \v(G_1) \geq \eps' \v(G_1)$ and $G':=G_1$ satisfies (ii).
	In the remaining case $\v(G_1) \geq Cd$, and by \cref{FindG'} applied with $s=d$, $n = \frac{\eps}{3C} d$  there exists a minor $G'$ of $G_1$ (and thus of $G$) such that $Cd \leq \v(G') \leq 2(C+1)d$, $\kappa(G')\geq \frac{\eps}{3C} d$ and $\delta(G')\geq (1 -\eps)d$. In particular, $$\kappa(G') \geq \frac{\eps}{3C} \frac{\v(G')}{2(C+1)} = \eps' \v(G'),$$ and so $G'$ satisfies (i). 
\end{proof}

Our final ingredient en route to the main result of this section establishes the existence of a small set of vertices in highly connected graphs that can be used to add edges to a minor found in the rest of the graph. It is inspired by \cite[Lemma 4.2]{Thomason01} and its application, although the argument in \cite{Thomason01} is more subtle.

\begin{lem} 
	\label{linkage}
	For every $\eps> 0$ there exists $\delta=\delta_{\ref{linkage}}(\eps)>0$ such that for every graph $G$ with  $\kappa(G)\geq \eps\,\v(G)$, there exists $Z\subseteq V(G)$ such that $|Z|\leq \eps \v(G)$ and for all $p_1,q_1,\dots p_t,q_t\in V(G)\setminus Z$ with $t \leq \delta\,\v(G)$, there exist a $(p_1q_1,\dots,p_tq_t)$-linkage $\mc{P}$ in $G$ such that all the internal vertices of paths in $\mc{P}$ lie in $Z$. 
\end{lem}

\begin{proof} 
	Assume $\eps < 1$ without loss of generality. Let $p := \frac{\eps}{2}$. Choose $\delta$ so that $ 0 < \delta < p^{\frac{1}{p}+2}/2$ and for every $n \geq 1/\delta$, 
	\begin{equation}\label{e:delta} 
	p^{\frac{1}{p}+ 1} n \geq 16 \log n. 
	\end{equation} 
Note that, in particular, $\delta \leq p/6$.
	
	Let $n := \v(G)$. Since the lemma is trivial for $n < 1/\delta$, we assume that $n \geq 1/\delta$. Let  $Z\subseteq V(G)$ be random with each vertex $v \in V(G)$ being added to $Z$ with probability $p$. By the Chernoff bound, 
	$$\brm{Pr}[|Z|\geq \eps \v(G)  ] \leq \exp\left(-\frac{pn}{3}\right) \leq \exp\left(-\frac{p }{3 \delta}\right) \leq \frac12.$$
	
	Say that a path in $G$ is \defn{short} if it has at most $1/p$ internal vertices. 
	Say that a pair of distinct vertices $u,v \in V(G)$ is \defn{good} if for some $t' \geq \delta n / p$  there exist internally disjoint short paths $Q_1,Q_2, \dots, Q_{t'}$ with ends $u$ and $v$ in $G$  such that all the internal vertices of these paths lie in $Z$. Note that if every pair of vertices of $G$ is good then  we can greedily construct the linkage $\mc{P}$ by selecting a short path from $p_i$ to $q_i$ disjoint from the previously selected paths. 
	
	It remains to show that a pair of  distinct vertices $u,v \in V(G)$ is not good with probability at most $\frac{1}{n^2}$, since then the union bound implies that with positive probability $|Z|\leq \eps\,\v(G)$  and every pair of  distinct vertices is good.
	
	There exist paths $P_1,P_2, \dots, P_{\kappa(G)}$  in $G$ with ends $u$ and $v$, and otherwise pairwise disjoint.  Since  $\kappa(G)/(2p) \geq n$  at least $\kappa(G)/2$ of these paths are short. The probability that all the internal vertices of a short path  lie in $Z$ is then at least $p^{1/p}$, and so the expected number $\mu$ of short paths among  $P_1,P_2, \dots, P_{\kappa(G)}$ that have all the internal vertices in $Z$ is at least $p^{1/p}\kappa(G)/2.$
	As $ p^{1/p}\kappa(G)/2 \geq p^{1/p+1}n \geq 2 \delta n/p$, 
	by the Chernoff bound the probability that the pair $(u,v)$ is not good is at most $$\exp\left(-\frac{\mu}{8} \right) \leq \exp\left(-p^{\frac{1}{p}+ 1} \frac{n}{8}\right) \stackrel{(\ref{e:delta})}{\leq} \frac{1}{n^2},$$
	as desired. 
\end{proof}

Finally, we combine Lemmas~\ref{l:copies}, ~\ref{c:denser} and~\ref{linkage} to obtain our main result. 

\begin{thm}\label{addedges} For all $\eps>0$ there exists $\delta=\delta_{\ref{addedges}}(\eps)>0$ such that for all $C > 0$ there exists  $L=L_{\ref{addedges}}(\eps,C)$ satisfying the following.
Let $J$ be a graph with $\v(J) \leq C$, let $\ell \geq L$ be an integer, and let $H$ be a graph such that $H \setminus F$ is a subgraph of $\ell\,J$ for some $F \subseteq E(H)$ with $|F| \leq \delta \v(H)$. Then $$c(H) \leq (1 + \eps)\ell\,c_f(J).$$ 	
\end{thm}

\begin{proof}
	Assume $\eps < \frac16$ without loss of generality. 
	Let $\eps' = \min\{\frac16,\eps'_{\ref{c:denser}}(6,\eps/2)\}$, $\delta=\delta_{\ref{linkage}}\s{\frac{1}{6}\eps\eps'}$, $\eps''= \frac{1}{6}\eps\eps'$, and  $L=\max_{J : \v(J) \leq C} L_{\ref{l:copies}}(J,\eps'').$ We show that these $\delta$ and $L$ satisfy the theorem.
	
	It suffices to show that  $H$ is a minor of every graph $G$ satisfying  $\d(G) \geq (1+\eps)\ell\,c_f(J)$. By \cref{c:denser} such a graph $G$  contains a minor $G'$ such that $\kappa(G') \geq  \eps' \v(G')$, and either
	\begin{itemize}
		\item[(i)] $\v(G') \geq 6\d(G)$ and $\delta(G') \geq (1 -\frac{\eps}{2})\d(G)$, or
		\item[(ii)] $\d(G') \geq (1 - \frac{\eps}{2})\d(G)$.
	\end{itemize}
	By the choice of $\delta$, there exists $Z \subseteq V(G')$ such that $|Z|\leq \frac{\eps\eps'}{6}\v(G')$ satisfying the conclusion of \cref{linkage}. Let $G''=G'\setminus Z$. It follows from the properties of $Z$ that if there exists a model of $\ell\,J$ in $G''$ then it can be extended to a model of $H$  in $G'$, and consequently $H$ is a minor of $G$.

	It remains to show that $\ell\,J$ is a minor of $G''$. We do this by verifying that the requirements of \cref{l:copies} are satisfied. Suppose first that (i) holds. Then $$\v(G'') \geq (1-\eps')\,\v(G') \geq 5\d(G) \geq 4(1+\eps'') \ell\,c_f(J) \geq 2(1+\eps'') \ell\,\v(J),$$ 
	and  
	 \begin{align*}\delta(G'') \geq\delta(G') - |Z|
	  & \geq  (1-\tfrac{\eps}{3} )(1 - \tfrac{\eps}{2} )\d(G) + \tfrac{\eps}{3}\kappa(G') -|Z| \\
	  &\geq \ell\, c_f(J) + \tfrac{\eps\eps'}{6}\v(G') \geq \ell \tau(J) + \eps''\v(G'') .
	\end{align*}
	Thus \cref{l:copies} (ii) is satisfied.	
	
	If (ii) holds then analogously $\d(G'') \geq \ell c_f(J)  + \eps''\v(G'')$,  and 	 \cref{l:copies} (i) is satisfied.
\end{proof}	

\cref{addedges} and \cref{l:newlower} imply that  $$\ell\, c_f(J) - 2 = \, c_f(\ell\,J) - 2  \leq c(\ell\,J) \leq (1+\eps)\ell\, c_f(J),$$
%\comment{KH: should this be:$\ell\, c_f(J) - 1 = \, c_f(\ell\,J) - 1  \leq c(\ell\,J) \leq (1+\eps)\ell\, c_f(J)$?}
for any graph $J$, $\eps >0$ and $\ell$ sufficiently large as a function of $J$ and $\eps$. Thus
$$c_f(J) = \lim_{\ell \to \infty}\frac{c(\ell\,J)}{\ell},$$
finally establishing the validity of the alternative definition \eqref{e:frac0} of the fractional extremal function, used in \cref{s:fracintro}.

\section{Decompositions}\label{s:decompose}

In this section we show that large graphs in  s.s.s.\ graph families are in a certain usable sense close to graphs with bounded maximum component size. This allows us to finish the proof of \cref{t:main1} deriving it from  \cref{addedges}. 

For a graph $G$, a collection $\mc{B}$ of subsets of $V(G)$ is a \defn{decomposition} of $G$ if for every $e \in E(G)$ there exists $B \in \mc{B}$ such that both ends of $e$ belong to $B$. Define the \defn{excess} of $\mc{B}$ as $\sum_{B \in \mc{B}}|B|- \v(G)$. We say that $\mc{B}$ is \defn{$C$-bounded} if $|B| \leq C$ for every $B \in \mc{B}$. The following lemma is well known \citep{Eppstein10}, although it has not been exactly stated as follows, so we include the proof for completeness. 

\begin{lem}
\label{Eppstein} 
Let $\mc{F}$ be a  graph family with strongly sublinear separators. Then for every $\eps>0$ there exists $C=C_{\ref{Eppstein}}(\mc{F},\eps)$ such that every graph  $G \in \mc{F}$ admit a $C$-bounded decomposition with excess at most $\eps\, \v(G)$.	
\end{lem}

\begin{proof}
We may assume that $\eps\leq 1$. Let $\beta < 1$ and $c >0$ be such that every graph $G\in \mc{F}$ has a separator of order at most $c\,{\v(G)^{\beta}}$.
Let $\gamma := c  \left(\left(\frac{1}{3}\right)^{\beta} + \left(\frac{2}{3}\right)^{\beta} - 1 \right)^{-1}$. We show that $C := 3 \left(\frac{\gamma}{\eps}\right)^{1/(1 - \beta)}$ satisfies the lemma.
We prove by induction on $\v(G)$ that every graph $G \in \mc{F}$ with $\v(G) \geq \frac{C}{3}$ has a $C$-bounded decomposition with excess at most $\eps \v(G) - \gamma{\v(G)^{\beta}}$. Clearly, this implies the lemma. In the base case, $\frac{C}{3} \leq \v(G) \leq C$, the trivial decomposition $\{V(G)\}$ satisfies the claim.

For the induction step, let $G$ be a graph in $\mc{F}$ with  $n:=\v(G) > C$. Let $(A_1,A_2)$ be a separation of $G$ with order at most $c n^{\beta}$, where $|A_1|,|A_2| \geq \frac{n}{3}$. By the induction hypothesis, for $i\in[2]$ there exists a  $C$-bounded decomposition $\mc{B}_i$ of $G[A_i]$ with excess at most $\eps |A_i| - \gamma|A_i|^{\beta}$. Then $\mc{B}_1 \cup \mc{B}_2$  is a  $C$-bounded decomposition of $G$ with excess at most
\begin{align*}
\eps (|A_1| +|A_2|) - &\gamma (|A_1|^{\beta}+|A_2|^{\beta}) \leq \eps n + c n^{\beta} - \gamma n^{\beta}\left(\left(\tfrac{1}{3}\right)^{\beta} + \left(\tfrac{2}{3}\right)^{\beta}\right)\\  &=
(\eps n - \gamma n^{\beta}) - \left(\left(\tfrac{1}{3}\right)^{\beta} + \left(\tfrac{2}{3}\right)^{\beta} - 1 -\tfrac{c}{\gamma} \right) \gamma n^{\beta} = 
\eps n - \gamma n^{\beta},
\end{align*}
as desired.
\end{proof}

It is convenient for us to replace large graphs with bounded maximum component size by many copies of the same bounded size graph. The next lemma does this.

\begin{lem}\label{lem:pigeonHole}\label{PH}
For all $C,\eps >0$ there exist $C'=C'_{\ref{PH}}(C,\eps)$ such that for every graph $H$ with maximum component size at most $C$ there exists a graph $J$ with $\v(J) \leq C'$ and $\ell\in\NN$ such that $H$ is isomorphic to an induced subgraph of $\ell\, J$, and $\v(\ell\,J) \leq (1 +\eps)\,\v(H)$. 
\end{lem}

\begin{proof} 
Assume without loss of generality that $\eps \leq 1$.
Let $\mc{G}=\{G_1,G_2,\dots,G_s\}$ be a set of representatives of all isomorphism classes of connected graphs on at most $C$ vertices. We show that $C' := \frac{3s C}{\eps}$ satisfies the lemma.
If $\v(H) \leq C'$ then $J=H$ and $\ell=1$ satisfies the lemma, and so we assume $\v(H) \geq C'$. Thus $sC \leq \frac{\eps}{3}\, \v(H)$. 
Let $a_i$ be the number of components of $H$ isomorphic to $G_i$. Let  $\ell := \lceil \frac{2 \v(H)}{C'} \rceil$, and let $J_i$ denote the disjoint union of $\lceil \frac{a_i}{\ell} \rceil$ copies of $G_i$, and let $J$ be obtained by taking the disjoint union of $J_1,J_2,\dots,J_s$. By construction, $H$ is isomorphic to an induced subgraph of $\ell\, J$. Moreover, \begin{align*}
	\ell \,\v(J) \leq \v(H)+\ell sC  \leq  \left(1 + \frac{2sC}{C'}\right)\,\v(H)+sC
	 &\leq  \left(1 + \frac{2sC}{C'} + \frac{\eps}{3}\right) \v(H) \\ 
	 & \leq (1+\eps)\,\v(H), 
	\end{align*} 
and $\v(J) \leq \frac{ 2\v(H)}{\ell} \leq C'$. Thus $\ell$ and $J$ satisfy all the requirements of the lemma.
\end{proof}

Combining \cref{Eppstein,PH}, we represent  large graphs in s.s.s.\ classes in a form which is amenable to the application of \cref{addedges}.  
	
\begin{lem}	
\label{ReduceToBoundedDecomp}\label{Reduce}
For  $\mc{F}$ be a graph family with strongly sublinear separators. Then for every $\eps >0$ there exists $C=C_{\ref{Reduce}}(\mc{F},\eps)$ such that for every graph $H \in \mc{F}$ there exists a graph $H'$  such that:
\begin{itemize}
	\item[(i)] $H$ is a minor of $H'$,
	\item[(ii)] there exists $F \subseteq E(H')$ with $|F| \leq \eps \v(H)$ and  $H' \setminus F$ is isomorphic to the graph $\ell\,J$ for some $\ell\in\NN$ and graph $J$ with $\v(J) \leq C$,
	\item[(iii)] $H' \setminus X$ is isomorphic to a subgraph of $H$ for some $X \subseteq V(H')$ with $|X| \leq \eps\, \v(H).$ 
\end{itemize} 
\end{lem}

\begin{proof} We assume $\eps \leq 1$ without loss of generality.
Let $C_1 :=C_{\ref{Eppstein}}(\mc{F},\frac{\eps}{4}).$ We show that $C :=C'_{\ref{PH}}(C_1,\frac{\eps}{3})$ 	satisfies the lemma.
By \cref{Eppstein} there exists a decomposition $\mc{B}$ of $H$ such that $|B| \leq C_1$ for every $B \in \mc{B}$ and $\sum_{B \in \mc{B}}|B|\leq (1+\frac{\eps}{4})\,\v(H)$. Let the graph $H_1$ be formed by taking disjoint union of the graphs $\{H[B]\}_{B \in \mc{B}}$. Let $H_2$ be obtained from $H_1$ by adding for every $v \in V(H)$ a set of edges $F_v$  joining distinct  vertices corresponding to $v$ in $H_1$ so that the resulting subgraph induced by these copies is a tree. Thus $|F_v|=|\{B \in \mc{B} : v\in B\}|-1$. Let $F := E(H_2)-E(H_1)$. Then 
\begin{equation}\label{e:ReduceF}
|F|=\sum_{v \in V(H)}|F_v|=\s{\sum_{B \in \mc{B}}|B|} - \v(H) \leq \frac{\eps}{4}\,\v(H).
\end{equation}
Note that $H$ can be obtained from $H_2$ by contracting the edges in $F$. Thus $H$ is a minor of $H_2$.
 
Since the maximum component size of $H_1$  is at most $C_1$, by \cref{PH} there exists $\ell\in\NN$ and a graph $J$ with $\v(J) \leq C$ such that $H_1$ is an induced subgraph of a graph  $H_3$ isomorphic to $\ell\,J$, and $\v(H_3) \leq (1 + \frac{\eps}{3} )\,\v(H_1)$. 

Finally, we show that $H'=H_2 \cup H_3$ satisfies the lemma. Condition (i) is satisfied since $H$ is a minor of $H_2$, and (ii) is satisfied by \eqref{e:ReduceF}. Let $X_1$ be the set of vertices of $H_1$ (or equivalently $H_2$)  corresponding to the vertices of $H$ that appear in at least two sets in $\mc{B}$. Let $X_2 := V(H_3)-V(H_1)$ and $X := X_1 \cup X_2$. Then $H' \setminus X = H_1 \setminus X_1$ is isomorphic to a subgraph of $H$
and 
$$|X|=|X_1| +|X_2| \leq  \frac{\eps}{2}\v(H) + \frac{\eps}{3}\s{1+\frac{\eps}{4}}\v(H) \leq \eps\,\v(H),$$
thus (iii) also holds.
\end{proof}

We now finish the proof of \cref{t:main1}, which we restate below for convenience, deriving it from \cref{addedges} and \cref{Reduce}.  

\Main*

\begin{proof} 
Given	 $0< \eps \leq 1,$  let $\delta :=\delta_{\ref{addedges}}(\frac{\eps}{4})$,  $C:=C_{\ref{Reduce}}(\mc{F},\delta)$ and let $N := C \cdot L_{\ref{addedges}}(\frac{\eps}{4},C)$. We show  that for every graph $H \in\mc{F}$ with $\v(H) \geq N$, 
$$c(H) \leq (1+\eps)c_f(H).$$
Clearly, this implies the theorem.

 By the choice of $C$, there exists graphs $H'$ and $J$ and there exists $\ell\in\NN$ and $X \subseteq V(H')$  satisfying \cref{Reduce} (i) and (ii) with $\eps$ replaced by $\delta$.
Since $$C \ell \geq \ell\,\v(J) =\v(H') \geq \v(H) \geq C \cdot L_{\ref{addedges}}(\tfrac{\eps}{4},C),$$ we have $\ell \geq L_{\ref{addedges}}( \frac{\eps}{2},C)$. Since $H' \setminus X$ is  isomorphic to a subgraph of $H$ for some $X \subseteq V(H')$ with $|X| \leq \frac{\eps}{4}\v(H),$ 
$$c_f(H) \geq c_f(H' \setminus X) \geq c_f(H') - |X| \geq c_f(H') - \tfrac{\eps}{4}\v(H) \geq c_f(H') - \tfrac{\eps}{2}c_f(H),$$
where the second inequality holds by \cref{lem:boundedDiff} and the last inequality holds since $c_f(H) \geq \frac{\v(H)}{2}$.
Thus $$c_f(H') \leq \s{1 + \tfrac{\eps}{2}}\,c_f(H).$$
Since $H' \setminus F$ is isomorphic to $\ell J$ for some $F \subseteq E(H')$ with $|F| \leq \delta \v(H)$,  \cref{addedges} is applicable to $H'$. Therefore   
\begin{align*}
c(H) \leq c(H') \leq \s{1+ \tfrac{\eps}{4}} \ell c_f(J) \leq \s{1+ \tfrac{\eps}{4}}c_f(H') & \leq  \s{1+ \tfrac{\eps}{4}} \s{1+ \tfrac{\eps}{2}}c_f(H) \\
& \leq (1 +\eps)c_f(H), 
\end{align*}
as desired.
\end{proof}

\section{Proof of \cref{t:upper0}}\label{s:upper}

In this section we prove \cref{t:upper0}, deriving it from \cref{t:main1} and \cref{c:tau}.
First, we give a general asymptotically tight upper bound on the ratio $\frac{c_f(H)}{\v(H)}$ in   s.s.s.\ graph families.

Let $\mc{F}$ be a graph class. Define the \defn{nucleus $\brm{nuc}(\mc{F})$} of $\mc{F}$ to be the set of all graphs $H$ such that $k\,H \in \mc{F}$ for every $k\in\NN$. Let 
$$ \rho(\mc{F}) := \sup_{H \in \brm{nuc}(\mc{F})}\frac{c_f(H)}{\v(H)}.$$

Note that $\mc{F}$ contains arbitrarily large graphs $G$ with $c(G)=(\rho(\mc{F}) -o(1))\,\v(G)$, since $\frac{c(k\,H)}{\v(k\,H)} \geq  \frac{c_f(H)}{\v(H)} -\frac{1}{k}$ for every non-null graph $H$ and $k\in\NN$. The next theorem shows that the above bound is tight for s.s.s.\ graph families $\mc{F}$; that is, $c(G) \leq (\rho(\mc{F})+o(1))\,\v(G)$ for every $G\in\mc{F}$. 

\begin{thm}\label{t:upper}
	For every s.s.s.\  graph family $\mc{F}$, 
	$$
	 \lim_{n \to \infty}\max_{\substack{G \in \mc{F} \\ \v(G)=n }}\frac{c(G)}{n} = \rho(\mc{F}). 
	$$
(In particular the above limit exists.)	
\end{thm}

\begin{proof}
	As mentioned above, it is not hard to see that $$\rho(\mc{F}) \leq \liminf_{n \to \infty}\max_{\substack{G \in \mc{F} \\ \v(G)=n }}\frac{c(G)}{n}.$$
	It remains to show that for every $\eps>0$ there exists $N$ such that 	
	\begin{equation}
	\label{e:upper}c(G) \leq (\rho(\mc{F})+\eps)\,\v(G)
	\end{equation} 
	for every $G \in \mc{F}$ with $\v(G) \geq N$. By \cref{t:main1}, we can replace $c(G)$ in \cref{e:upper} by $c_f(G)$.
	Let $C:=C_{\ref{Eppstein}}(\mc{F},\frac{\eps}{2})$. Let $k_0$ be chosen so that  if $k_0 J \in \mc{F}$ for some graph $J$ with $\v(J) \leq C$ then  $J\in \brm{nuc}(\mc{F})$. Let $M$ be the number of isomorphism classes of graphs on at most $C$ vertices.  
	Let $n :=\v(G)$. By the choice of $C$ there exists $X \subseteq V(G)$ with $|X| \leq \frac{\eps}{2}\v(G)$ such that $G \setminus X$ has maximum component size at most $C$. By the choice of $M$, $k_0$ and $N$, it further follows that the total size of components of $G \setminus X$ that do not belong to the nucleus of  $\mc{F}$  is less than $k_0M <\frac{\eps}{2}\v(G).$ Thus there exists $Y \subseteq V(G)$ with $|Y| \leq \eps\,\v(G)$ such that 	every component of $G \setminus Y$  belongs to $\brm{nuc}(\mc{F})$. By \cref{lem:boundedDiff} and the subadditivity of $c_f$, 
	\begin{align*}
	c_f(G) \leq c_f(G \setminus Y)+ |Y| \leq \rho(\mc{F})\,\v(G \setminus Y) + |Y| \leq \rho(\mc{F}+\eps)\,\v(G),
	 \end{align*}
	 as desired.
\end{proof}

In the case   when $\tau(H) \geq \frac{2}{3}\v(H)$ for some $H \in \brm{nuc}(\mc{F})$, \cref{c:tau} yields a particularly simple formula for $\rho(\mc{F})$. 

\begin{cor}\label{c:rho}
	Let  $\mc{F}$  be a monotone  graph family. Let 
	$$r := \sup_{H \in \brm{\brm{nuc}(\mc{F})}}\frac{\tau(H)}{\v(H)}.$$
	If $r \geq \frac23 $ then $\rho(\mc{F}) = r$.  	
\end{cor}
\begin{proof}
	 By \cref{cfLowerBound} we have
	$\rho(\mc{F}) \geq r$.
	On the other hand,  \cref{c:tau} implies $c_f(H) \leq r \v(H)$ for every $H \in \brm{nuc}(\mc{F})$, and so 	$\rho(\mc{F}) \leq r$.
\end{proof}

Note that \cref{t:upper} and \cref{c:rho} immediately imply \cref{t:upper0}.

\smallskip

Finally,  we derive from \cref{c:rho} a formula for $\rho(\mc{F})$ for every proper minor-closed graph family which is equivalent to an open weakening of Hadwiger's conjecture, as discussed in the introduction.

Let $\mc{C}_t$ denote the class of all graphs with maximum component size at most $t$.

\begin{thm}\label{thm:minorclosed}
The following statements are equivalent for any integer $t\geq 3$:
	\begin{itemize}
		\item[(i)] $\alpha(H) \geq \frac{\v(H)}{t-1}$ for every non-null $K_t$-minor free graph $H$, and
		\item[(ii)] $\rho(\mc{F}) \leq \frac{t-2}{t-1}$ for every minor-closed class of graphs $\mc{F}$ such that $\mc{C}_t \not \subseteq \mc{F}.$ 
	\end{itemize}
\end{thm}

\begin{proof}
	{\bf (ii) $\Rightarrow$ (i)}. Suppose that (i) does not hold. Let $H$ be a graph such that $\alpha(H) < \frac{\v(H)}{t-1}$ and $K_t$ is not a minor of $H$. 
	Let $\mc{F}$ consist of all graphs $G$ such that every component of $G$ is a minor of $H$. Then $\mc{F}$ is minor-closed, and $K_t \not \in \mc{F}$ implying that $\mc{C}_t \not \subseteq \mc{F}.$ On the other hand, $H \in \brm{nuc}(\mc{F})$ and so $$\rho(\mc{F}) \geq \frac{c_f(H)}{\v(H)} \geq \frac{\v(H) - \alpha(H)}{\v(H)}> \frac{t-2}{t-1}.$$ 
	Thus (ii) does not hold, as desired.
	
\smallskip
	{\bf (i) $\Rightarrow$ (ii)}.
	Let $\mc{F}$ be minor-closed class of graphs with $\mc{C}_t \not \subseteq \mc{F}.$  If $\mc{C}_3 \not \subseteq \mc{F}$ then every graph in $H \in \brm{nuc}(\mc{F})$ is a forest. Therefore $c_f(H)=\frac{1}{2}\v(H)$ for every such $H$, implying $\rho(\mc{F}) = \frac{1}{2}$, as desired.
	Now assume that $\mc{C}_3 \subseteq \mc{F}.$ Thus $K_3 \in \brm{nuc}(\mc{F})$.  Let $$r := \sup_{H \in \brm{\brm{nuc}(\mc{F})}}\frac{\tau(H)}{\v(H)}.$$
	Then $r \geq \frac{\tau(K_3)}{\v(K_3)} \geq \frac23 $, and $\rho(\mc{F}) = r$ by \cref{c:rho}. Since $\mc{C}_t \not \subseteq \mc{F}$, every graph $H \in \brm{\brm{nuc}(\mc{F})}$ is $K_t$-minor free, and so $\tau(H) \leq \frac{t-2}{t-1}\v(H)$ for every such $H$ by (i). This implies $\rho(\mc{F}) = r \leq \frac{t-2}{t-1}$, as desired.   
\end{proof}

In summary, \cref{t:upper,thm:minorclosed} jointly imply that if  Hadwiger's conjecture (or its weakening \cref{thm:minorclosed} (i)) holds for  $K_t$-minor free graphs, then for every minor-closed class $\mc{F}$ such that 
$\mc{C}_{t-1} \subseteq \mc{F}$  and $\mc{C}_{t} \not \subseteq \mc{F},$  
$$c(G) \leq \frac{t-2}{t-1}\,\v(G) + o(\v(G)),$$ for every graph $G \in \mc{F}$.
Moreover, the coefficient $\frac{t-2}{t-1}$ cannot be improved.

%%%%%%%%%%%%%%%%%%%%%%
\section{General bounds}\label{s:general}

In this section we extend our investigation beyond graph families with strongly sublinear separators and consider families with unbounded density. We show that the maximum density of families of regular graphs with the extremal function linear in the number of vertices is logarithmic. In particular, we show that the extremal function of hypercubes is linear. We also prove upper and lower bounds on the extremal function of regular graphs with density slightly above the logarithmic threshold. We make no attempt to optimize the constant coefficients in this section.

We use two additional external tools in our proofs. The first one allows us to only look for minors in very dense graphs, more specifically in graphs $G$ with density substantially larger than $\frac{\v(G)}{2}$, at a cost of constant factor loss in the bounds on the extremal function. It is  due to Reed and Wood~\cite{ReeWoo15} and extends an earlier similar result by Mader~\cite{Mader68}.

\begin{lem}[{\cite[Lemma 8]{ReeWoo15}}]\label{lem:ReedWood}
	For every integer $k > 1$, every graph $G$ with  $\d(G) \geq 2k$ contains a  minor $G'$  such that $\delta(G') \geq \max\{k-1, 0.64\cdot \v(G')\}$, $2\delta(G') - \v(G') > 0.46k$, and $\v(G') \leq 4k.$
\end{lem}

Secondly, we need a powerful theorem of Lee~\cite{LeeRamsey17} that guarantees the existence of prescribed bipartite subgraphs in sufficiently large and dense host graphs.

\begin{thm}[{\cite[Theorem 1.3]{LeeRamsey17}}]\label{t:Lee0}
	Let  $d,n \geq 2$ be integers, and let $\alpha, \eps>0$ be real such that $\alpha^{d(d-2)} \leq \eps < 1$. Let $G$ be a graph such that $\v(G) \geq (1+\eps)\alpha^{-d}n$ and $\d(G) \geq \frac{\alpha}{2}\v(G)$, and let $H$ be a bipartite graph on $n$ vertices with bipartition $(A,B)$ such that
	$\deg(v) \leq d$ for every $v \in A$ and $$\frac{|B|^d}{|B|(|B|-1)\cdots(|B|-d+1)} \leq 1+\eps.$$ 
	Then $G$ contains a subgraph isomorphic to $H$.
\end{thm}

\begin{cor}\label{t:Lee}
Let  $d \geq 2$ be an integer, and let $H$ be a  bipartite graph with bipartition $(A,B)$ such that
$\deg(v) \leq d$ for every $v \in A$. Then every graph $G$ with $\v(G) \geq (\v(H)+d^2)2^{d+1}$ and $\d(G) \geq \frac{\v(G)}{4}$  contains a subgraph isomorphic to $H$.
\end{cor}
\begin{proof} Let $H'$ be a  bipartite graph with bipartition $(A,B')$ obtained from $H$ 
	By adding at most $d^2$ isolated vertices to $B$ so that $|B'| \geq d^2$.  Then
	\begin{align*}
	\frac{|B'|^d}{|B'|(|B'|-1)\cdots(|B'|-d+1)} &\leq  \fs{|B'|^2}{|B'|(|B'|-d+1)}^{d/2} = \s{1+\frac{d-1}{|B'|-d+1}}^{d/2} \\ &\leq e^{\frac{d}{2} \cdot \frac{d-1}{|B'|-d+1}} \leq \sqrt{e}, 
	\end{align*}
	and $\v(G) \geq  2\cdot 2^{d} \cdot \v(H')$.
Thus 	 $G$ and $H'$ satisfy the conditions of \cref{t:Lee0} with $\alpha = 1/2$, $\eps = \sqrt{e}-1$, and $n =\v(H')$. It follows, that  $G$ contains a subgraph isomorphic to $H'$, and consequently a  subgraph isomorphic to $H$, as desired.
\end{proof}

It is easy to extend \cref{t:Lee} to graphs with bounded maximum component size and a looser bound on the number of vertices.

\begin{cor} \label{c:Lee}
Let $d,C\in\NN$ and let $H$ be a bipartite graph with bipartition $(A,B)$ such that $\deg(v) \leq d$ for every $v \in A$ and the maximum component size of $H$ is at most $C$. Then every graph $G$ with $\v(G) \geq C2^{d+1}+ d^2 2^{d+1} + \v(H)$ and $\d(G) \geq \frac{\v(G)}{4} + \v(H)$  contains a subgraph isomorphic to $H$.
\end{cor}	

\begin{proof}
	Let $H'$ be a maximal subgraph of $H$ consisting of a union of connected components  of $H$ such that $G$ contains a subgraph isomorphic to $H'$. Suppose for a contradiction that $H' \neq H$. Let $J$ be a component of $H \setminus V(H')$ and let $G'$ be obtained from $G$ by deleting the vertex set of the subgraph of $G$ isomorphic to $H'$.  Then $\v(G') \geq \v(G) - \v(H) \geq C2^{d+1}+ d^2 2^{d+1}$, and   $\d(G') \geq \d(G) - \v(H) \geq \frac{\v(G)}{4} \geq \frac{\v(G')}{4}$.  By \cref{t:Lee},  $J$ is isomorphic to a subgraph of $G'$, implying that $H' \cup J$ is isomorphic to a subgraph of $G$, in contradiction to the choice of $H'$.
\end{proof}

\cref{lem:ReedWood} and \cref{c:Lee} imply the following upper bound.

\begin{lem} \label{l:bipcomp}
Let $d,C\in \NN$ with $C \geq d^2$. Let $H$ be a graph and let $H'=H\setminus F$ be a spanning subgraph of $H$ for some $F \subseteq E(H)$ with $|F| \leq \v(H)$.   If $H'$ is bipartite and admits a bipartition $(A,B)$ such that $\deg(v) \leq d$ for every $v \in A$, and  the maximum component size of $H'$ is at most $C$, then $$c(H)=O(\v(H) + C2^d).$$
\end{lem}	

\begin{proof} 
	Let $k := 10(\v(H)+C2^d)$. It suffices to show that every graph $G$ with  $\d(G) \geq 2k$ contains $H$ as a minor. 
	By \cref{lem:ReedWood}, $G$ contains a minor $G'$ such that $\delta(G') \geq   \max\{k-1, 0.64 \v(G')\}$, $2\delta(G') - \v(G') > 0.46k$, and $k \leq \v(G') \leq 4k.$  Since  $$\d(G') \geq \frac{\delta(G')}{2} \geq \frac{\v(G')}{4} + 0.23 k \geq \frac{\v(G')}{4} +\v(H),$$
	and 
	$$\v(G') \geq k \geq \v(H) + C2^{d+1}+ d^2 2^{d+1},$$  the graph $G'$ contains $H'$ as a subgraph by \cref{c:Lee}. Since $$\v(H) +|F|< 0.46k < 2\delta(G') - \v(G'),$$ by greedily adding 2-edge paths corresponding  to edges of $F$  to this subgraph, we can extend it to the desired subdivision of $H$ in $G'$.
\end{proof}

Extending \cref{l:bipcomp} to our main result requires an additional straightforward lemma that shows that every graph is a minor of a bipartite graph with small degrees on one side of the bipartition and not too many vertices.

\begin{lem}\label{l:degenerate}
	Let  $\Delta \geq d \geq 2$ be positive integers. Let $H$ be a $\Delta$-degenerate graph. There exists a bipartite graph $H'$ with bipartition $(V(H),W)$ such that $\deg(v) \leq d$ for every $v \in W$, $H \preceq H'$ and $$|W| = \left\lceil \frac{\Delta}{d-1}\right\rceil\v(H).$$
\end{lem}
\begin{proof}
	Let $v_1,\dots,v_n$ be an ordering of $V(H)$ such that each $v_i$ has at most $\Delta$ neighbours in the set $\{v_j \mid j < i\}$ and let $\ell:=\left\lceil \frac{\Delta}{d-1}\right\rceil$. Let $H^*$ be obtained from $H$ by adding for each $i\in [v(H)]$ an independent set $W_i:=\{w_{i,j}\mid j\in [\ell]\}$ of new vertices, each of degree at most $d$, such that each is adjacent to $v_i$ and the union of their neighbourhoods covers the neighbours of $v_i$ in $\{v_j \mid j < i\}$. Now $H':=H^*-E(H)$ has the desired properties.
\end{proof}

The following theorem is the main result of this section. It provides an upper bound on the extremal functions of sparse, easily decomposable graphs. 

\begin{thm}\label{t:general} Let $\Delta \geq 3$, $m \geq 0$ and $C \geq \Delta^2$ be integers. Let  $H$ be a $\Delta$-degenerate graph such that $\v(H) > C$, and let $\mc{B}$ be a $C$-bounded decomposition of $H$ with excess at most $m$. 
Then \begin{equation}\label{e:general}
c(H) = O\s{\v(H)+ m + \frac{\Delta\cdot \v(H)}{\log(\v(H)/C)}}.
\end{equation}
\end{thm}
\begin{proof}	Let $$d := \min\s{\Delta,  \left\lfloor \log_2\fs{\v(H)}{C} \right \rfloor + 3}.$$
By \cref{l:degenerate} for each $B \in \mc{B}$, there exists a bipartite graph  $H'_B$  such that $H'_B$ admits a bipartition such that the maximum degree of vertices in one of the parts is at most $d$,  $H[B]$ is a minor of $H'_B$, and  $$\v(H'_B) = \s{\left\lceil \frac{\Delta}{d-1} \right\rceil +1 }\v(H[B]) \leq C\s{\left\lceil \frac{\Delta}{d-1} \right\rceil +1 }.$$
Let  $H'$ be obtained from the disjoint union of graphs $\{H'_B\}_{B \in \mc{B}}$ by adding  a set  $F$ of edges with $|F| \leq \min\{m,\v(H')\}$ joining the vertices corresponding to the same vertex of $H$ so that $H$ is a minor of $H'$. 
By \cref{l:bipcomp}, %changed from {l:degenerate}
$$c(H') = O\s{\v(H') + \frac{\Delta}{d}C 2^d} = O\s{\v(H') + \frac{\Delta\cdot \v(H)}{\log(\v(H)/C)}}.$$
Since $\v(H')=O\s{\v(H)+ m + \frac{\Delta\cdot \v(H)}{\log(\v(H)/C)}}$, the theorem follows.
\end{proof}

\cref{t:general} immediately implies the following.

\begin{cor}\label{c:component} 
For every graph $H$ and integer $k \geq 2$, 
\begin{equation}
\label{e:component}
\frac{c(k\,H)}{\v(k\,H)} = O\s{1+ \frac{\v(H)}{\log{k}}}.
\end{equation}
\end{cor}
\begin{proof}
If  $k \leq (\v(H))^3$, then the claim  holds since $$\frac{c(k\,H)}{\v(k\,H)} \leq \frac{c(K_{\v(k\,H)})}{\v(k\,H)} = O(\sqrt{\log\v(k\,H)}) = O(\sqrt{\log\v(H)}) = O\s{1+ \frac{\v(H)}{\log{k}}}.$$	
Now assume $k \geq (\v(H))^3$.
Apply \cref{t:general} to $k\,H$ with $\Delta = \v(H)$, $C = \v^2(H)$ and $m=0$. By our assumption on $k$, we have $\log(\frac{\v(k\,H)}{C}) = \Omega(\log k)$, and thus \eqref{e:general} implies \eqref{e:component}.
\end{proof}	

As proved in the previous sections, 
$$1 > \frac{c_f(H)}{\v(H)}  = \lim_{k \to \infty}\frac{c(k\,H)}{\v(k\,H)}.$$
\cref{c:component} sheds some light on the speed of convergence of the sequence towards its limit.
In particular, it shows that $c(k\,H)$ is linear in $\v(k\,H)$ when $k$ is exponential in $\v(H)$. 
Thus, for example, graphs in the family $\{2^t \cdot K_t \}$ have logarithmic density and linear extremal function. Another corollary of \cref{t:general} shows that the family of hypercubes has the same properties.

\begin{cor}
\label{c:hypercube} 
Let $Q_d$ be the $d$-dimensional hypercube. Then $$c(Q_d) = \Theta(2^d).$$
\end{cor}

\begin{proof}
The lower bound $c(Q_d)\geq 2^{d-1}-1$ follows from \cref{NaiveLowerBound}. For the upper bound, let $(F_1,F_2)$ be a partition of $E(Q_d)$ into two sets, such that the edges of $F_1$ are parallel to the first $\lceil \frac{d}{2} \rceil$ standard basis vectors in the natural embedding of $Q_d$ in $\bb{R}^d$, and the edges of  $F_2$ are parallel to 	the remaining $\lfloor \frac{d}{2} \rfloor$ basis vectors. Let $H_i$ be the spanning subgraph of $Q_d$ with $E(H_i)=F_i$ for $i\in[2]$, and let $\mc{B}$ be the collection of vertex sets of components of $H_1$ and of $H_2$. Then $\mc{B}$ is a $2^{\lceil d/2\rceil}$-bounded decomposition of $Q_d$ with excess $2^d$. Since $Q_d$ is $d$-degenerate, the result follows  from \cref{t:general} applied with $\Delta=d$, $m=2^d$ and $C=2^{\lceil d/2\rceil}$.
\end{proof}
	
The next lemmma complements \cref{c:component} by providing a lower bound on the extremal function of regular graphs that is superlinear in the number of vertices for graphs with superlogarithmic density. 

\begin{lem}\label{l:genlower} 
For every $d\in\NN$ and for every $d$-regular graph $H$ with $d \geq \log \v(H)$, 
$$c(H) = \Omega\s{  \v(H)\sqrt{\log \fs{d}{\log \v(H)}}}.$$ 
\end{lem}

\begin{proof}
Let $t:=\v(H)$.	Since $c(H) \geq \frac{\v(H)}{2} -1 $ for every graph $H$, the lemma holds in the regime $\frac{d}{\log t} = O(1)$, and so we may assume without loss of generality that $\frac{d}{\log t} \geq 8$. Since $d < t$, in particular, $t \geq 9$.

Let $G=G(n,\frac12)$ be the Erd\H{o}s--Renyi random graph with vertex set $[n]$, where each pair of distinct vertices are adjacent with probability $\frac12$. 
	
We upper bound the probability that a fixed map $\mu:V(H) \to \mc{P}([n])$, mapping vertices of $H$ to pairwise disjoint subsets of $V(G)$ is a model of $H$. Let $$X = \left\{v \in V(H) \,|\, |\mu(H)| \leq \frac {3n}t \right\}.$$ Then $|X| \geq \frac{2t}{3}$, and $\e(H[X]) \geq \frac{dt}{2} - d(|V(H)-X|) \geq \frac{dt}{6}$. 

For each edge $uv \in E(H[X])$ the probability that there exists an edge of $G$ with one end in $\mu(u)$ and the other in $\mu(v)$ is
$$ 1-2^{-|\mu(v)||\mu(u)|} \leq 1-2^{-9n^2/t^2}  \leq \exp \s{-2^{-9n^2/t^2}}.$$	
Thus the probability that $\mu$ is a model is at most
$$ \prod_{uv \in E(H[X])}\s{1-2^{-|\mu(v)||\mu(u)|}} \leq  \exp \s{\e(H[X]) \cdot 2^{-9n^2/t^2}} \leq \exp \s{\frac{dt}{6} 2^{-9n^2/t^2}}.$$
Since there exist at most $t^n$ maps $\mu$ as above, the probability that $H$ is a minor of $G$ is at most 
\begin{equation}\label{e:exp}
\exp \s{n\log t - \frac{dt}{6}\, 2^{-9n^2/t^2}}.
\end{equation}
Since $\d(G) \geq (n-1)/2$ with probability at least $\frac12$, it  suffices to show that  if $n =\eps t\sqrt{\log_2  \fs{d}{\log t}}$	for some $\eps < \frac16$ then the expression in \eqref{e:exp} is at most $\frac12$.
Substituting the above formula for $n$ using $\eps$ into \eqref{e:exp}, and defining $x := d/\log{t} $ for brevity, 
\begin{align*}
\exp& \s{n\log t - \frac{dt}{6} 2^{-9n^2/t^2}} = \exp \s{t\log t \s{\eps\sqrt{\log_2  x} - \frac{1}{6}x^{1-9\eps^2}} } \\ & \leq \exp \s{ \frac{1}{6} t\log t (\sqrt{\log_2{x}} - x^{7/16})} \leq  \exp \s{ \frac{\sqrt{3} - 2}{6} t\log t} \leq \frac1e, 
\end{align*} 	
as desired, where the last two inequalities follow since $x \geq 8$ and $t \geq 9$.
\end{proof}

\cref{l:genlower} combined with the upper bound of \citet{ThoWal19} implies that for every $\eps >0$ and every graph $d$-regular graph $G$ with $d \geq \log^{1+\eps}\v(H)$, 
$$c(H) = \Theta_{\eps}\s{\v(H)\sqrt{\log d}}.$$
 
Note that the assumption on regularity in the above lower bounds is necessary. For example, as shown by Kapadia, Norin and Qian~\cite{KNQ21}, 
$$c(K_{s,t})= \Theta(\sqrt{st \log s}+s+t)$$
for all integers $t \geq s \geq 2$. Thus  the extremal function of graphs $K_{s, \lfloor s \log s \rfloor}$ is linear in the number of vertices, while they have much higher density than any  regular graphs with this property.

\subsection*{Acknowledgements} This work was partially completed while SN was visiting Monash University in 2019 supported by a Robert Bartnik Visiting Fellowship. SN thanks the School of Mathematics at Monash University for its hospitality.

We thank J\'er\'emie Turcotte for valuable comments.

%%%%%%%%%%%%%  Squashing the bibliography 
%\fontsize{10pt}{10pt}
\let\oldthebibliography=\thebibliography
\let\endoldthebibliography=\endthebibliography
\renewenvironment{thebibliography}[1]{%
	\begin{oldthebibliography}{#1}%
		\setlength{\parskip}{0ex}%
		\setlength{\itemsep}{0ex}%
	}{\end{oldthebibliography}
}

\appendix

\section{Derivation of \cref{emb}}

Let $G$ be a graph.
For disjoint sets $X,Y \subseteq V(G)$, let  \defn{$\e_G(X,Y)$} denote the number of edges of $G$ between $X$ and $Y$. The \defn{density} of $(X,Y)$ is defined as $\d_G(X,Y) :=\frac{\e_G(X,Y)}{|X||Y|}.$ A pair $(A,B)$ of disjoint subsets of $V(G)$ is \defn{$\eps$-regular} if $|\d_G(X,Y) - \d_G(A,B)| \leq \eps$ for all $X \subseteq A$, $Y \subseteq B$ such that $|X| \geq \eps|A|$ and $|Y| \geq \eps |B|$. The following degree version of the regularity lemma is the first of the two ingredients we use to derive \cref{emb}.

\begin{thm}[{\cite[Theorem 1.10]{KomSim93}}]\label{t:reg1}
	For every $\eps >0$ there exists $T=T_{\ref{t:reg1}}(\eps)$ such that for any graph $G$ and any $d>0$ there exists a partition $(V_0,V_1,\dots,V_{l})$ of $V(G)$ and a spanning subgraph $G'$ of $G$ satisfying the following:
	\begin{itemize}
		\item $l \leq T$,
		\item $|V_0| \leq \eps \v(G)$,
		\item $|V_1|=|V_2|=\dots=|V_l| \leq \lceil \eps \v(G) \rceil$,
		\item $\deg_{G'}(v) > \deg_G(v) - (d+\eps)\,\v(G)$ for every $v \in V(G)$,
		\item $\e(G'[V_i])=0$ for all $i \geq 1$,
		\item the pair $(V_i,V_j)$ is $\eps$-regular in $G'$ with $\d_{G'}(V_i,V_j) \geq d$, or $\d_{G'}(V_i,V_j) =0$ for all $1 \leq i < j \leq l$.
	\end{itemize} 
\end{thm}

Let $\mc{V}=(V_0,V_1,\dots,V_l)$ be  a partition of $V(G)$ satisfying the conclusion of \cref{t:reg1} for given $d,\eps >0$, and let $k=|V_1|=|V_2|=\dots=|V_l|$. We associate with $\mc{V}$ a graph $R=R_d(G,\mc{V})$ with $V(R)=[l]$, where $\{i,j\} \in E(R)$ if and only if  $\d_{G'}(V_i,V_j) \geq d$. We say that $R$ is a \defn{$(d,k,\eps)$-regular reduction} of $G$.  
The second ingredient we need is the following embedding lemma.

\begin{thm}[{\cite[Theorem 2.1 and a follow-up remark]{KomSim93}}]\label{t:emb1}
For all $d>0$ there exists $\eps'_0=\eps'_{\ref{t:reg1}}(d) >0$ such that for every $0 < \eps' < \eps'_0$ there exists $\eps_0=\eps_{\ref{t:reg1}}(d,\eps')>0$ satisfying the following. Let $G$ be a graph, and let   $R$ be a $(d, k',\eps)$-regular  reduction of $G$ for  some $0 < \eps \leq \eps_0$. Let $k \leq (1-\eps')k'$ be a positive integer, and  let $H$  be a subgraph of $R^{(k)}$ such that every component of $H$ has at most $\eps'k$ vertices. Then $H$ is isomorphic to a subgraph of $G$.
\end{thm}

We now derive \cref{emb}, which we restate for convenience, from \cref{t:reg1,t:emb1}.

\Emb*

\begin{proof}
	Let $d := \frac{\eps}{7}$, let $\eps'_0 := \eps'_{\ref{t:reg1}}(d)$, let $\eps' := \min (\frac{\eps}{7}, \eps'_0)$, let $\eps_0 := \eps_{\ref{t:reg1}}(d,\eps')$, and let $\eps_1 := \min (\frac{\eps}{7}, \eps_0)$. By the choice of $d$, $\eps'$ and $\eps_1$, 
	 $$d+2\eps_1+4\eps' \leq \eps.$$
	We prove the theorem for $$T:=T_{\ref{t:reg1}}(\eps_1) \qquad \mathrm{and}\qquad N: = \frac{KT}{\eps'(1-2\eps')(1-\eps_1)}.$$
	
	Let $G$ be a graph with $\v(G) \geq N$. Let  $\mc{V}=(V_0,V_1,\dots,V_l)$ be a partition of $V(G)$ and let $G'$ be a subgraph of $G$ satisfying the conditions of \cref{t:reg1} applied with $\eps_1$ in place of $\eps$. Let $R=R_d(G,\mc{V})$ be a $(d, k',\eps_1)$-regular  reduction of $G$ associated with $\mc{V}$. Then $\v(R)  = l \leq  T$, as desired. Note also that $k'\v(R)=\v(G)-|V_0|$, and so $(1-\eps_1)\,\v(G) \leq k'\v(R) \leq \v(G).$
	
	Let $k:= \lfloor (1-\eps')k' \rfloor \geq (1-2\eps')k'$, where the last inequality holds since $$\eps'k' \geq \eps'(1-\eps_1)\frac{\v(G)}{T} \geq \eps'(1-\eps_1)\frac{N}{T} \geq 1$$ by the choice of $N$. Thus $k\v(R) \leq k'\v(R) \leq \v(G),$ and $$k\v(R) \geq (1-2\eps')k'\v(R) \geq (1-\eps_1)(1-2\eps')\,\v(G) \geq (1 -\eps)\,\v(G).$$
	Thus the first condition of the theorem holds.
	
	 Let $G'':=G' \setminus V_0$. Note that $G''$ is a spanning subgraph of $R^{(k')}$. Thus $$k'\delta(R) \geq \delta(G'') \geq \delta(G') - |V_0| \geq \delta(G) - (d+2\eps_1)\,\v(G).$$
	This in turn implies
	$$k \delta(R) \geq (1-2\eps')k'\delta(R)  \geq  \delta(G) - (d+2\eps_1+2\eps')\,\v(G) \geq   \delta(G) - \eps\,\v(G),$$
	verifying the second condition of the theorem.
	
	Similarly,
	\begin{align*}(k')^2 \e(R) = \e(R^{(k')}) \geq\e(G'')& \geq \e(G) - |V_0|\cdot \v(G) - (d+\eps_1)\v^2(G) \\ & \geq \e(G) - (d+2\eps_1)\v^2(G),\end{align*}
	and so 
	$$k^2 \e(R) \geq (1-2\eps')^2 (k')^2  \e(R) \geq \e(G) - (d+2\eps_1+4\eps')\v^2(G) \geq   \e(G) - \eps\,\v^2(G),$$
	verifying the third condition.

Finally, by the choice of $\eps_1$ to satisfy the conditions of \cref{t:reg1}  every subgraph $H$  of $R^{(k)}$ with maximum component size at most $\eps'k$ is isomorphic to a subgraph of $G$. Since $$\eps'k \geq \frac{\eps'(1-2\eps')(1-\eps_1)}{T}\v(G) \geq \eps'(1-2\eps')(1-\eps_1)\frac{N}{T} = K,$$ by the choice of $N$, the last condition of the theorem holds.
\end{proof}

\begin{aicauthors}
	\begin{authorinfo}[kevin]
		Kevin Hendrey\\
		Discrete Mathematics Group\\
		Institute for Basic Science\\
		Korea\\
		kevinhendrey\imageat{}ibs\imagedot{}re\imagedot{}kr\\
		\url{https://sites.google.com/view/kevinhendrey}
	\end{authorinfo}
	\begin{authorinfo}[sergey]
		Sergey Norin\\
		Department of Mathematics and Statistics\\
		McGill University\\
		Montr\'eal, Canada\\
		snorin\imageat{}math\imagedot{}mcgill\imagedot{}ca\\
		\url{https://www.math.mcgill.ca/snorin/}
	\end{authorinfo}
	\begin{authorinfo}[david]
		David R. Wood\\
		School of Mathematics\\
		Monash University\\
		Melbourne, Australia\\
		david.wood\imageat{}monash\imagedot{}edu\\
		\url{https://users.monash.edu.au/~davidwo/}
	\end{authorinfo}
\end{aicauthors}
\end{document}